\documentclass{birkau}

\usepackage{amsmath,amssymb,url}
\usepackage{enumitem} 
\usepackage{color} 
\usepackage{tikz-cd}
\usepackage{mathtools}
\usepackage{quiver}
\usepackage{graphicx} 
\usepackage{hyperref}
\usepackage{nicematrix} 

\tikzcdset{scale cd/.style={every label/.append style={scale=#1},
    cells={nodes={scale=#1}}}}

\numberwithin{equation}{section}

\theoremstyle{plain}
\newtheorem{theorem}{Theorem}[section]
\newtheorem{lemma}[theorem]{Lemma}

\newtheorem{observation}[theorem]{Observation}
\newtheorem{corollary}[theorem]{Corollary}
\newtheorem{claim}{Claim}[theorem]

\newtheorem*{THMACCFGTHM}{Lemma~\ref{thm:ACCFG}}
\newtheorem*{THMdownsetsofNmTHM}{Lemma~\ref{thm:downsetsofNm}}
\newtheorem*{THMpropCRTHM}{Theorem~\ref{thm:propCR}}
\newtheorem*{THMgaussrearrangementTHM}{Lemma~\ref{thm:gaussrearrangement}}
\newtheorem*{THMdecompositionTHM}{Lemma~\ref{thm:decomposition}}
\newtheorem*{THMcanon45THM}{Lemma~\ref{thm:canon45}}
\newtheorem*{THMcanon2THM}{Lemma~\ref{thm:canon2}}
\newtheorem*{THMcanon3THM}{Lemma~\ref{thm:canon3}}
\newtheorem*{THMdisjeq01THM}{Lemma~\ref{thm:disjeq01}}
\newtheorem*{THMcanon12THM}{Lemma~\ref{thm:canon12}}

\theoremstyle{definition}
\newtheorem{definition}[theorem]{Definition}
\newtheorem{convention}[theorem]{Convention}
\newtheorem{remark}[theorem]{Remark}

\DeclareMathOperator{\cls}{cls}
\DeclareMathOperator{\Clo}{Clo}
\DeclareMathOperator{\sClo}{sClo}
\DeclareMathOperator{\upset}{\uparrow}
\DeclareMathOperator{\downset}{\downarrow}
\DeclareMathOperator{\Pol}{Pol}
\DeclareMathOperator{\Inv}{Inv}
\DeclareMathOperator{\CSP}{CSP}
\DeclareMathOperator{\sPol}{sPol}

\newcommand{\N}{\mathbb{N}}
\newcommand{\OO}{\mathcal{O}}
\newcommand{\pp}[2][]{\left[#2\right]^{#1}_{\mathrm{pp}}\,}
\newcommand{\qpp}[2][]{\left[#2\right]^{#1}_{\mathrm{qpp}}\,}
\newcommand{\eo}[2]{\left[#1\right]_{(\mathrm{eo}1)-(\mathrm{eo}#2)}\,}
\newcommand{\eoc}[1]{\left[#1\right]_{(\mathrm{eo}6)}\,}
\newcommand{\KR}{K\!R}
\newcommand{\CR}{C\!R}
\newcommand{\disj}[1]{\left\vert\!\left\vert #1 \right\vert\!\right\vert}

\begin{document}


\title[Multi-sorted relational clones on $\{0,1\}$]{On the lattice of multi-sorted relational clones on a two-element set}



\author[V. David]{Vojtěch David}
\address{Department of Theoretical Computer Science and Mathematical Logic,\\ Charles University, Czechia}
\email{vojtech.david@matfyz.cuni.cz}

\author[D. N. Zhuk]{Dmitriy N. Zhuk}
\address{Department of Algebra,\\ Charles University, Czechia}
\email{dmitrii.zhuk@matfyz.cuni.cz}

\thanks{Both authors were supported by the Czech Science Foundation project 25-16324S. The work of the first author is part of the project SVV-2025-260837.
The second author is also funded by the European Union (ERC, POCOCOP, 101071674). Views and opinions expressed are however those of the author(s) only and do not necessarily reflect those of the European Union or the European Research Council Executive Agency. Neither the European Union nor the granting authority can be held responsible for them.}




\begin{abstract}
We introduce a new approach to the description of multi-sorted clones (sets of $k$-tuples of operations of the same arity, closed under coordinatewise composition and containing all projection tuples) on a two-element domain. Leveraging the well-known Galois connection between operations and relations, we define a small class of \textit{canonical relations} sufficient to describe all Boolean multi-sorted clones up to non-surjective operations. Furthermore, we introduce \textit{elementary operations} on relations, which are less cumbersome than general formulas and have many useful properties. Using these tools, we provide a~new and elementary proof of the famous Post's lattice theorem. We also show that every multi-sorted clone of $k$-tuples of operations decomposes into a surjective part described by canonical relations and $2k$ clones of $(k-1)$-tuples of operations. This structural understanding allows us to describe an embedding of the lattice of multi-sorted clones into a well-understood poset. In particular, we rederive -- by a simpler method -- a~result of V.~Taimanov originally from 1983, showing that every multi-sorted clone on a two-element domain is finitely generated. Finally, we also give a~concise proof of the Galois connection between (surjective) multi-sorted clones and the corresponding closed sets of relations.
\end{abstract}

\maketitle

\section{Introduction and main results}\label{sec:intro}
\subsection{Introduction}

Clones are sets of operations closed under composition and containing all projections; objects that naturally arise in most areas of mathematics. For instance, we know that composing monotone operations yields a monotone operation, 
and composing linear operations yields a linear one, which gives rise to the notions of clones of monotone and linear operations, respectively.  
Arguably the main result in clone theory is the description of all clones on a 2-element set obtained by E. Post in \cite{Post1920,Post1941}. This result is not only elegant and important on its own -- description of all clones is essentially the description of all algebraic structures -- but has become a valuable tool in areas ranging from logic to complexity theory \cite{CSPAndPOST}.
In light of this result, a natural next step was to seek similar descriptions for larger domains. However, it seems unrealistic to describe all clones even on a three-element set as the lattice is uncountable \cite{JanovMucnik1959,HulanickiSwierczkowski1960} and very complicated \cite{bulatov1999sublatticesone,
bulatov1999sublatticestwo,
bulatov2001conditions,MoiseevNUclones}.

Probably the second most important result is a Galois connection between clones and relational clones \cite{Geiger1968,Bodnarchuk1969I,Bodnarchuk1969II}, where relational clones are sets of relations (or, equivalently, predicates, see Definition~\ref{def:relations}) closed under primitive positive (pp for short) formulas (that is, formulas using only $\wedge$ and $\exists$) and containing both empty and equality relations. Through this correspondence, every clone on a finite set can be described as the set of all operations preserving some (possibly infinite) set of relations. More importantly, the entire lattice of clones can be analyzed purely in terms of relations and pp formulas.

One key advantage of this perspective is that working with relations tends to be more tractable than working directly with operations, since pp definitions are often simpler to compute than compositions. 
In fact, describing minimal clones (the atoms of the lattice of clones) requires dealing with operations and their composition, and to date, they have only been classified for three-element domain \cite{csakany1983minimal}, and even those were found by a computer. A~classification for the four-element domain was announced in \cite{MinimalCLonesOn4Elements}, but to the~best of our knowledge, it has not been fully confirmed.
In contrast, maximal clones (the coatoms of the lattice) on any finite set have been fully classified for over 50 years \cite{rosmax}, with
each maximal clone being defined by a~concrete relation describing some nice property of operations such as being monotone or linear.

In this paper, we consider a generalization of clones known as multi-sorted clones, 
where instead of operations, we work with $k$-tuples of operations of the same arity (referred to as $k$-operations in the paper).
Each operation in the tuple then corresponds to a specific sort, 
and composition is defined coordinatewise (see the precise definition in Section \ref{sec:definitions}).

Various classification results are already known in the multi-sorted setting when the domain is fixed to $\{0,\,1\}$. It is known that there are only countably many $k$-sorted clones on a two-element domain and that each is finitely generated -- a result first announced by V. Taimanov in 1983 \cite{Taimanov1983} and later established in detail in his 1984 PhD thesis \cite{Taimanov1984thesis}. More recently, Taimanov republished these results in a series of three shorter papers \cite{Taimanov2019,Taimanov2020,Taimanov2022}. The~works \cite{Taimanov1983,Taimanov1984thesis} also include a classification of the maximal multi-sorted clones on $\{0,\,1\}$, with alternative description given in \cite{Romov1987,Romov1989}. A related result -- the classification of maximal $k$-sorted clones where the first $k-1$ sorts consist of operations on two-element set and the last on any arbitrary set -- is presented in \cite{Romov1991}, and the maximal multi-sorted clones with operations on either a two- or three-element set are described in \cite{Marchenkov1994}. Finally, a classification up to minion homomorphisms of multi-sorted clones on $\{0,\,1\}$ determined by binary relations was recently given in \cite{Barto2025}. The main aim of our paper is to continue this classification program by providing new tools and structural insights.

One reason the multi-sorted generalization is particularly natural is the Galois connection we can build between clones of $k$-operations and relational clones of $k$-sorted relations; that is, relations where each variable has assigned a sort from 1 to $k$ and variables may only be identified if they share the same sort.
In fact, in both theory and applications, it is common for different coordinates of a relation to come from different domains, making multi-sorted relations natural objects in their own right.

Another important feature of working with relations is that we do not need all of them to define clones.
Knowing the Galois connection, we can argue that it suffices to consider relations that cannot be represented as conjunctions of relations of smaller arities -- such relations are called \textit{essential}~\cite{zhuk_dm_post,CardinalityAboveMinimalZhuk}. 
Already, this simple observation led to a new proof of Post's lattice theorem \cite{zhuk_dm_post} and was the main idea behind the description of all clones below the maximal clone of self-dual operations on three-element domain  \cite{mvlsc},  
which, to this day, remains the only uncountable sublattice below a maximal clone that has been described.

A very natural next step is to consider only relations that 
are $\cap$-irreduci\-ble in the corresponding relational clone,  called \textit{critical} in~\cite{agnesParal} and \textit{maximal} in~\cite{mvlsc}.
However, since being critical heavily depends on the relational clone and is therefore hard to verify, it is more desirable to introduce a combinatorial property (similar to essentiality) that is almost as strong as criticality. Such a property, the notion of key relation, was introduced in \cite{ZhukKeyRelations}, and it is equivalent to being critical in the relational clone generated by the relation, assuming that all the coordinates are of different sorts.
Another motivation for studying key relations is
a nice description of all (multi-sorted) key relations on the two-element set:
they are exactly the relations that can be defined as disjunctions of linear equations \cite{ZhukKeyRelations}. Moreover, as explained in detail in Section \ref{sec:definitions},
computing primitive positive formulas over such relations is very simple.
In the paper, we exploit these disjunctive relations to describe multi-sorted relational clones 
on a two-element set.

Nevertheless, in order to characterize all multi-sorted clones on a two-element domain, we still need to consider all 
combinations of such relations -- and there are simply too many of them. 
To overcome this, we take one more step: we strengthen our closure operator by allowing conjunctive formulas with both existential and universal quantifiers (called quantified primitive positive formulas, or qpp formulas for short; see Section~\ref{sec:definitions} for the precise definition).
This closure operator also admits a Galois connection, but this time, we consider only surjective operations instead. 
Fortunately, on a two-element domain, the only nonsurjective operations are constants, which can be easily added back to the surjective part afterward.
To make the consideration of all combinations feasible, we introduce \textit{canonical relations} -- a restricted class of relations satisfying the following three properties:
(1)~any quantified relational clone can be defined by these relations,
(2)~for any canonical relation $\rho$ there are no weaker relations that altogether generate $\rho$, 
(3)~the arity of the relations cannot be reduced while preserving property (1). 
To demonstrate the gap between arbitrary relations and canonical relations, notice that on a two-element set, there are $2^{2^n}\cdot k^{n}$ $n$-ary $k$-sorted relations, 
about $(1-1/\sqrt{e})\cdot 2^{2^n}\cdot k^{n}$ $n$-ary $k$-sorted essential relations \cite{Zharikov}, 
about $2^{n^2/4}\cdot k^{n}$ $n$-ary $k$-sorted key relations, 
and only about $2^n\cdot k^n$ $n$-ary $k$-sorted canonical relations. In fact, for $n >k$, the only $n$-ary $k$-sorted canonical relations are the relations of size $2^n-1$; see Definition~\ref{def:canonicalrelations} for details. 

Let us note that a similar approach -- restricting attention to carefully chosen relations -- was already taken by B.~Romov in \cite{Romov1987,Romov1989,Romov1991} and by S.~Marchenkov in \cite{Marchenkov1994}. Interestingly, some of the relations they considered are very similar to key and canonical relations, even though their work predates the theory introduced in \cite{ZhukKeyRelations} by more than 20 years. This highlights how natural this perspective is and further motivates the need for a unified framework.

\subsection{Main results}

Since clones on a two-element set are fully characterized, while clones on larger domains are too complicated to describe,
it is natural to attempt a~description of multi-sorted clones on a two-element domain.
In this paper, we develop a technique that makes this feasible.

The first main result of the paper is presented in Section~\ref{sec:canpreds}, where we define canonical relations and prove that they satisfy the properties (1)--(3) outlined above.
Further, in Section~\ref{sec:eo}, we introduce elementary operations and show how they can be used instead of qpp formulas for relations defined as disjunctions of linear equations.
These two ideas together are so powerful that deriving Post's lattice from them becomes straightforward, as demonstrated in Section~\ref{sec:post}.
Thus, a new short proof of Post's lattice is another main result of the paper.

The third result states that any clone of $k$-operations can be decomposed into a surjective part definable by canonical relations, which therefore admits an embedding into a simple and well-understood poset, 
and $2k$ clones of $(k-1)$-operations.  
This decomposition implies Taimanov's result \cite{Taimanov1983,Taimanov1984thesis,Taimanov2019,Taimanov2020,Taimanov2022} -- that there are only countably many $k$-sorted clones and that each of them is finitely generated. 
Note that Taimanov worked mostly with operations, whereas in our setting operations are auxiliary objects that allow us to make decomposition -- and theoretically, they could be avoided altogether.

The last contribution of this paper is a complete proof of a Galois connection 
between clones of $k$-sorted surjective operations and quantified $k$-sorted relational clones. 
Although our proof is essentially a generalization of the corresponding result from \cite{Borner2009} to the multi-sorted setting, to our knowledge, it has never been properly written down.

\subsection{Applications}

Even though we do not provide a complete description of multi-sorted clones on a two-element domain even for two sorts, it is plausible to derive such a description using the ideas developed in this paper. 
Below, we give four other examples of how our results can be applied to concrete problems. 

\begin{enumerate}
\item An important class of clones in universal algebra 
consists of clones of idempotent operations, that is, operations satisfying $f(x,\dots,\,x) = x$. 
For instance, the entire theory of Taylor varieties and Taylor identities was developed primarily for the idempotent reduct of algebras \cite{taylor1977varieties,maroti2008existence}. 
Relational clones corresponding to idempotent clones are precisely those containing all constant relations. Since any qpp formula using universal quantifiers may be replaced by a pp formula using these constant relations, our technique provides a natural tool for characterizing idempotent clones.
\item One of the minimal clones on $\{0,\,1,\,2\}$ is generated by 
the following semiprojection~\cite{MinimalCLonesOn4Elements}:
\begin{equation*}
s_5(x,\,y,\,z) \coloneqq \begin{cases}
x, & |\{x,\,y,\,z\}|<3,\\
y, & |\{x,\,y,\,z\}|=3.\\
\end{cases}
\end{equation*}
It was shown in~\cite[Section 6]{CardinalityAboveMinimalZhuk} that there are two types of
essential relations on $\{0,\,1,\,2\}$
preserved by $s_5$:
graphs of permutations and
relations whose projection onto every coordinate is a two-element set.
The~latter relations can be viewed as multi-sorted relations with variables of three sorts,
where sorts depend on the projections onto the corresponding coordinate, that is  $\{0,\,1\}$, $\{1,\,2\}$, or $\{0,\,2\}$.
Therefore, clones on three elements containing this semiprojection can be embedded into the clones of 3-operations on a two-element domain 
and it seems plausible that our technique could lead to a complete characterization of all such clones.
\item Primitive positive formulas play an important role in the study of the complexity of the \textit{Constraint Satisfaction Problems} (CSP). 
Let $\Gamma$ be a~set of multi-sorted relations on a two-element set,  and let $\CSP(\Gamma)$ denote a decision problem where the input 
is a conjunction of relations from $\Gamma$, and the task is to decide whether it is satisfiable.  
Since checking satisfiability is equivalent to evaluating a sentence after adding existential quantifiers, 
$\CSP(\Gamma_1)$ is LOGSPACE-reducible to $\CSP(\Gamma_2)$  whenever $\Gamma_2$ pp-defines $\Gamma_1$ \cite{jeavons1,jeavons3}. 
Our results imply that instead of considering arbitrary $\Gamma$ we may consider $\Gamma$ consisting of disjunction of linear equations 
and our technique can be further developed into an algorithm for CSP that eliminates existential quantifiers one by one using elementary operations. Furthermore, it is known that CSP over an arbitrary $\Gamma$ is 
equivalent to the CSP over its \textit{core} \cite{jeavons1}, 
and qpp formulas over a core can be replaced by equivalent pp formulas. Therefore, 
for the purpose of studying complexity
of CSP for constraint languages consisting of 
multi-sorted relations on a two-element set, it suffices to assume that $\Gamma$ consists only of canonical relations. For more information about CSP and its connection with clones, see \cite{Barto2017}.
\item Evaluating a quantified conjunctive sentence is another important problem in computer science,  called \textit{quantified} CSP \cite{QC2017}. 
Similarly to the~non-quantified case, the complexity depends only on the quantified relational clones generated by the constraint language $\Gamma$ \cite{BBCJK}, and therefore we may assume that
all the relations in $\Gamma$ are canonical. As canonical relations are very simple and there are only few of them, this substantially simplifies the task of characterizing the complexity for all sets of multi-sorted relations on a two-element domain.

\end{enumerate}

\subsection{Structure of the paper}

The paper is organized as follows.

In Section~\ref{sec:definitions}, we fix notation and provide all necessary preliminary definitions. In particular, we formally introduce multi-sorted (operational) clones, relational clones, and quantified relational clones, and recall the standard notions of polymorphisms and invariant relations that will be essential throughout the text.

Section~\ref{sec:canpreds} is dedicated to canonical relations. We explicitly define canonical relations, state their properties, and outline the reasoning why each quantified relational clone on a two-element domain can be generated using only canonical relations.

In Section~\ref{sec:closednessandpost}, we introduce elementary operations on relations. These operations provide explicit and convenient closedness criteria for quantified relational clones. As an immediate consequence, we obtain a simplified proof of Post's lattice theorem.

In Section~\ref{sec:lattice}, we investigate the structure of the lattice of multi-sorted (quantified) relational clones. We establish embeddings of this lattice into simpler and better-understood structures, and use these embeddings to study some properties of multi-sorted clones. In particular, we revisit the aforementioned result concerning finite basis of multi-sorted clones, simplifying its original proof by V. Taimanov.

Section~\ref{sec:Galois} contains a rigorous treatment of the Galois connection between clones of multi-sorted surjective operations and quantified relational clones, thus generalizing classical single-sorted results.

Finally, Section~\ref{sec:defproofs} provides detailed proofs of technical lemmata stated in previous sections.

\section{Main definitions and preliminaries}\label{sec:definitions}
We begin by fixing some notation and defining some key concepts that will be used throughout the paper. Some of the definitions are adapted from \cite{ZhukKeyRelations}.

\begin{convention}[Main conventions]
Throughout the text, we use boldface to denote tuples.

By $\N$ we denote the set of natural numbers without $0$. By $\N_0$, we denote the set $\N\cup\{0\}$. For $k\in\N$ by $E_k$ we denote the set $\{1,\dots,\,k\}$.

For a set $A$ and $n\in\N$, by $2^A$, we denote the power set of $A$, and by~$A^n$, we denote the $n$-th cartesian power of $A$.

If a function has many arguments, a tuple has many entries, or a set has many elements, we may sometimes write them on multiple lines enclosed in parentheses; e.g., the following notations are equivalent:
\begin{equation*}
    f\left(\begin{array}{@{\,}l@{\,}}x_1,\dots,\,x_n,\\
    y_1,\dots,\,y_m\end{array}\right)\quad\text{and}\quad f(x_1,\dots,\,x_n,y_1,\dots,\,y_m)\,.
\end{equation*}

We use the standard notation for logical connectives. The symbols $\top$ and $\bot$ represent \textit{true} and \textit{false}, respectively. Finally, we assume that conjunctions ($\land$) and disjunctions ($\lor$) take precedence over quantifiers, allowing us to omit parentheses when the grouping is unambiguous. For example, we write
\begin{equation*}
    \exists y \, \rho(x,\,y)\land \sigma(w,\,z,\,y) \quad\text{instead of}\quad \exists y \, \left(\rho(x,\,y)\land \sigma(w,\,z,\,y)\right)\,.
\end{equation*}
\end{convention}
\begin{definition}[Operations, clones]\label{def:operations}
Let $A$ be a~nonempty finite set and $k,\,n\in\N$. By a~\textit{$k$-operation} of \textit{arity} $n$, we denote a~$k$-tuple of functions $A^n\to A$. Let $\OO_{A,n}^k$ be the~set of all $n$-ary $k$-operations on $A$ and let $\OO^k_A = \bigcup_{m\,\in\,\N}\OO_{A,m}^k$ be the set of all finitary $k$-operations on $A$. For a~$k$-operation $\boldsymbol{f}\in \OO_{A,n}^k$, the corresponding tuple of functions is $(\boldsymbol{f}^{(1)},\,\boldsymbol{f}^{(2)},\dots,\, \boldsymbol{f}^{(k)})$. We say that a~$k$-operation $\boldsymbol{f}$ is surjective if $\boldsymbol{f}^{(i)}$ is a surjective operation for every $i\in E_k$.

We define a~\textit{composition} of $k$-operations in a~natural way. For $\boldsymbol{f}\in\OO_{A,m}^k$ and $\boldsymbol{g}_1,\dots,\,\boldsymbol{g}_m\in\OO_{A,n}^k$, we define a new $k$-operation $\boldsymbol{h}\in\OO_{A,n}^k$ by putting
\begin{equation*}
    \boldsymbol{h}^{(i)}(x_1,\dots,\,x_n) \coloneqq \boldsymbol{f}^{(i)}(\boldsymbol{g}_1^{(i)}(x_1,\dots,\,x_n),\dots,\,\boldsymbol{g}_m^{(i)}(x_1,\dots,\,x_n))
\end{equation*}
for every $i\in E_k$. We sometimes use the shorter notation
\begin{equation*}
    \boldsymbol{h}(x_1,\dots,\,x_n) = \boldsymbol{f}(\boldsymbol{g}_1(x_1,\dots,\,x_n),\dots,\,\boldsymbol{g}_m(x_1,\dots,\,x_n))\,.
\end{equation*}

A~$k$-\textit{clone} on $A$ is a set of $k$-operations closed under composition, which, for each $n\in\N$ and $i\in E_n$, contains the \textit{projection $k$-operations} (or \textit{$k$-projections}) $\boldsymbol{p}_{n,i}^k$ where
\begin{equation*}
    {\boldsymbol{p}_{n,i}^k}^{(j)}(x_1,\dots,\,x_n) = x_i
\end{equation*}
for all $j\in E_k$. Finally, for a set $\mathcal{F}\subseteq \OO^k_A$ we define $\Clo^k\mathcal{F}$ as the smallest $k$-clone on $A$ containing~$\mathcal{F}$ and we put $\sClo^k\mathcal{F} \coloneqq \{\boldsymbol{f}\in\Clo^k \mathcal{F}\mid \boldsymbol{f}\text{ is surjective}\}$. Any set $\mathcal{C}\subseteq\OO_A^k$ satisfying $\sClo^k\mathcal{C} = \mathcal{C}$ is called a \textit{surjective clone}.
\end{definition}
\begin{definition}[Relations, relational clones]\label{def:relations}
A~\textit{predicate} on a~nonempty finite set $A$ is a~mapping $A^n \to \{\top,\,\bot\}$ where $n\in\N_0$ is called the~\textit{arity} of the~predicate. Let $k\in\N$. A~$k$-\textit{sorted} predicate is a~predicate in which every variable has a~sort from the~set $E_k$. When necessary, we denote the sorts of variables by superscript. By $R_A^k$ we denote the set of all $k$-sorted predicates on~$A$.

By $\sigma_\bot$, we denote the predicate of arity $0$, which takes the value $\bot$. Furthermore, for each $i\in E_k$, we define the predicate
\begin{equation*}
    \sigma_{=}^{i}(x,\,y) = \top \quad\iff\quad x=y\,,
\end{equation*}
where $x$ and $y$ are variables of the $i$-th sort.

Suppose $S\subseteq R_A^k$; then by $\pp[k]{S}$, we denote the set of all $k$-sorted predicates $\rho$ such that
\begin{equation*}
    \rho(x_1,\dots,\,x_n) = \exists y_1\dots \exists y_l\, \rho_1(v_{1,1},\dots,\,v_{1,n_1})\land\dots\land\rho_m(v_{m,1},\dots,\,v_{m,n_m})\,,
\end{equation*}
where
\begin{equation*}
    \rho_1,\dots,\,\rho_m\in S\cup\{\sigma_\bot,\,\sigma_{=}^1,\dots,\,\sigma_{=}^k\}
\end{equation*}
and $v_{i,j}\in\{x_1,\dots,\,x_n,\,y_1,\dots,\,y_l\}$ are variable symbols that are subject to the following restriction: if a variable is substituted in some predicate as a variable of the $i$-th sort, then this variable cannot be substituted in any predicate as a variable of another sort. That is, $\rho$ is defined by a~first-order formula that uses only predicates in $S\cup\{\sigma_\bot,\,\sigma_{=}^1,\dots,\,\sigma_{=}^k\}$, conjunction, and existential quantification, with the sorts of all variables adhering to the described condition -- such formula is called a~\textit{sorted primitive positive formula} (shortly, pp formula).

Equivalently, we define a closure under \textit{sorted quantified primitive positive formulas} (shortly, qpp formulas): by $\qpp[k]{S}$, we denote the set of all predicates defined from elements of $S$ by a first-order formula that uses only predicates in $S\cup\{\sigma_\bot,\,\sigma_{=}^1,\dots,\,\sigma_{=}^k\}$, conjunction, existential quantification, and universal quantification, with the sorts of all variables, once again adhering to the described condition.

A~set $S\subseteq R^k_A$ is called
\begin{itemize}
    \item $k$-sorted \textit{relational clone} if $\pp[k]{S}=S$;
    \item $k$-sorted \textit{quantified relational clone} if $\qpp[k]{S} = S$.
\end{itemize}

For simplicity, we usually omit brackets and write $\pp[k]{\rho_1,\dots,\,\rho_n}$ and $\qpp[k]{\rho_1,\dots,\,\rho_n}$ instead of $\pp[k]{\{\rho_1,\dots,\,\rho_n\}}$ and $\qpp[k]{\{\rho_1,\dots,\,\rho_n\}}$ respectively.

We say that two predicates are \textit{similar} if they differ only by an order of variables.

Finally, we do not distinguish between $n$-ary predicates and $n$-ary \textit{relations}, that is, subsets of $A^n$ with coordinates of corresponding sorts. For a~relation $\rho$ and a~tuple $\boldsymbol{a}$, notations $\rho(\boldsymbol{a}) = \top$ and $\boldsymbol{a} \in \rho$ are equivalent. We will mostly use the relational formalism, as it is arguably more established; however, we will occasionally also rely on the predicate formalism for the sake of precision.
\end{definition}
\begin{convention}
In some limited cases, like $\OO_A^k$ or $R_A^k$, we use the superscript as a part of a fixed symbol to denote the ``sorting'' of operations or relations instead of cartesian powers. These symbols will always be defined before they are used.
\end{convention}
\begin{remark}
It is also common \cite{BulatovMultisorted2003,Zhuk2020} to define multi-sorted relations and operations by considering distinct domains for each sort. In this approach, a~multi-sorted operation of arity $n$ consists of a tuple of functions $f_i: A_i^n \to A_i$, while a multi-sorted relation is a subset of $A_1 \times \dots \times A_k$, where $A_1, \dots, A_k$ are finite but possibly different domains. However, these definitions are essentially equivalent to our setting, as any such operation can be viewed as an element of $\OO_{A_1\, \cup\, \cdots\, \cup\, A_k}^k$, and any such relation can be viewed as an element of $R_{A_1\, \cup\, \cdots\, \cup\, A_k}^k$.
\end{remark}
\begin{definition}[Polymorphisms, invariant relations]
Let $k\in\N$ and let $\rho$ be a~$k$-sorted relation of arity $n$ on a nonempty finite set $A$ and let $r_i$ be the sort of its $i$-th variable for every $i\in E_n$. We say that an $m$-ary $k$-operation $\boldsymbol{f}$ on $A$ \textit{preserves}~$\rho$ if
\begin{equation*}
    \boldsymbol{f} \left(\!\begin{pmatrix}
    a_{1,1} \\
    a_{2,1} \\
    \vdots\\
    a_{n,1}
    \end{pmatrix},\,\begin{pmatrix}
    a_{1,2}\\
    a_{2,2}\\
    \vdots \\
    a_{n,2}
    \end{pmatrix},\dots,\,\begin{pmatrix}
    a_{1,m} \\
    a_{2,m} \\
    \vdots \\
    a_{n,m}
    \end{pmatrix}\!\right)\coloneqq \begin{pmatrix}
    \boldsymbol{f}^{(r_1)}(a_{1,1},\, a_{1,2}, \dots,\, a_{1,m}) \\
    \boldsymbol{f}^{(r_2)}(a_{2,1},\, a_{2,2}, \dots,\, a_{2,m}) \\
    \vdots \\
    \boldsymbol{f}^{(r_n)}(a_{n,1},\, a_{n,2}, \dots,\, a_{n,m})
    \end{pmatrix} \in \rho
\end{equation*}
for any $(a_{1,1},\dots,\,a_{n,1}),\dots,\,(a_{1,m},\dots,\,a_{n,m})\in\rho$. Based on this notion, we define the following mappings between the sets of operations and the sets of relations:
\begin{align*}
    \Pol^k S &\coloneqq \{\boldsymbol{f}\in\OO_A^k \mid \forall \rho\in S:\boldsymbol{f}\textrm{ preserves }\rho\}\,, \\
    \sPol^k S &\coloneqq \{\boldsymbol{f}\in\OO_A^k \mid (\forall \rho\in S:\boldsymbol{f}\textrm{ preserves }\rho)\text{ and }(\boldsymbol{f}\textrm{ is surjective})\}\,, \\
    \Inv^k \mathcal{F} &\coloneqq \{\rho\in R_A^k \mid \forall \boldsymbol{f}\in\mathcal{F}:\boldsymbol{f}\textrm{ preserves }\rho\}\,,
\end{align*}
for any $S\subseteq R_A^k$ and $\mathcal{F}\subseteq \OO_A^k$. The images of these mappings are respectively called \textit{polymorphisms},\textit{ surjective polymorphisms}, and \textit{invariant relations}.
\end{definition}
\begin{convention}
Whenever the meaning is clear from the context, we simplify the terminology by omitting the ``$k$-sorted'' and refer simply to \textit{relations}, \textit{relational clones}, and \textit{quantified relational clones}.

Similarly, we will simplify the notation of operators $\Clo^k$, $\sClo^k$, $\pp[k]{-}$, $\qpp[k]{-}$, $\Pol^k$, $\sPol^k$, $\Inv^k$ by omitting the $k$ in superscript as it does not cause any ambiguity.
\end{convention}

Independently studied in \cite{Geiger1968} and \cite{Bodnarchuk1969I,Bodnarchuk1969II}, there is a well-known Galois connection between the sets of $1$-sorted relations and the sets of $1$-operations given by the mappings $\Pol$ and $\Inv$ where the closed sets are relational clones (pp closed sets of relations) and clones. Later in \cite{Romov1973} it was proved that this result also extends to $k$-sorted relations and $k$-operations.

Similarly, there is a Galois connection between the sets of $1$-sorted relations and $1$-operations given by the mappings $\sPol$ and $\Inv$. It was proved in \cite{Borner2009} that the closed sets of this Galois connection are quantified relational clones and surjective clones. As with the standard $\Pol$-$\Inv$ Galois connection, we may extend this result for $k$-sorted relations and $k$-operations. We will prove the following general statement in Section \ref{sec:Galois}.
\begin{theorem}\label{thm:Galois}
Let $k\in\N$ and let $A$ be a~nonempty finite set. For all $\mathcal{F}'\subseteq\mathcal{F}\subseteq \OO_A^k$ and all $S'\subseteq S \subseteq R_A^k$, we have
\begin{align*}
    \Inv \mathcal{F} &\subseteq \Inv\mathcal{F}'\,, & \mathcal{F}&\subseteq\sPol\Inv\mathcal{F}\,, & \Inv\sPol\Inv\mathcal{F} &= \Inv\mathcal{F}\,, \\
    \sPol S &\subseteq \sPol S'\,, & S&\subseteq\Inv\sPol S\,, & \sPol\Inv\sPol S &= \sPol S\,,
\end{align*}
and
\begin{align*}
    \sPol\Inv \mathcal{F} &= \sClo\mathcal{F}\,,\\
    \Inv\sPol S &= \qpp{S}\,.
\end{align*}
\end{theorem}

Even though this theorem holds for an arbitrary finite set, in the rest of the paper, we will only consider the~two-element domain $A = \{0,\,1\}$.

A particular consequence of this theorem that will be useful in Section~\ref{sec:lattice} is that the mappings $\Inv$ and $\sPol$ give mutually inverse, order-reversing bijections between the sets of surjective clones and the sets of quantified relational clones (ordered by inclusion).

Next, we introduce a class of relations that will serve as a~convenient tool for characterizing (quantified) relational clones.

\begin{definition}[Linear equations]
An expression
\begin{equation*}
    a_1 x_1 + \dots + a_n x_n = b
\end{equation*}
where $a_1,\dots,\,a_n,\,b\in\{0,\,1\}$, $x_1,\dots,\,x_n$ are variables, and addition and multiplication are performed modulo $2$ is called a~\textit{linear equation}. Fixing the~order of variables, any linear equation may be viewed as a predicate that evaluates to $\top$ if and only if the equation is satisfied. We say that a relation (/predicate) $\rho$ may be represented as a~\textit{disjunction of linear equations} if and only if
\begin{equation*}
    \rho(x_1,\dots,\,x_n) = \disj{\begin{array}{c}
        \lambda_1(x_1,\dots,\,x_n)  \\
        \vdots \\
        \lambda_m(x_1,\dots,\,x_n)
    \end{array}} \coloneqq \lambda_1(x_1,\dots,\,x_n)\lor \dots\lor \lambda_m(x_1,\dots,\,x_n)
\end{equation*}
for some linear equations $\lambda_1,\dots,\,\lambda_m$. The notation $\disj{-}$ will henceforth denote a relation formed as the disjunction of the listed relations, as shown above.

By $\KR^k$, we denote the set of all relations from $R_{\{0,\,1\}}^k$ that may be represented as disjunctions of linear equations.
\end{definition}
The notation $\KR^k$ comes from the fact that these relations are precisely the \textit{key relations} on $\{0,\,1\}$; see \cite{ZhukKeyRelations} for details. It was shown in \cite[Section 2.2]{ZhukKeyRelations} that every relational clone is uniquely determined by the relations that may be represented as a~disjunction of linear equations.
\begin{theorem}\label{thm:KR}\cite[Section 2.2]{ZhukKeyRelations}
Let $k\in\N$ and $S\subseteq R_{\{0,\,1\}}^k$. Then
\begin{equation*}
    \pp{S} = \pp{\pp{S}\cap \KR^k}\,.
\end{equation*}
\end{theorem}
In Section~\ref{sec:canpreds}, we introduce an even smaller class of relations that retains this property in the context of quantified relational clones.

We conclude this section with the definition of dummy variables and a~final convention to improve the overall clarity.
\begin{definition}
Let $A_1,\dots,\,A_n,\,B$ be arbitrary sets, and let $f: A_1\times \dots \times A_i \times \dots \times A_n \to B$ be some mapping. We say that the $i$-th variable of $f$ is \textit{dummy} if for all $(a_1,\dots,\,a_n)\in A_1\times\dots\times A_n$ and for all $a_i'\in A_i$, we have
\begin{equation*}
    f(a_1,\dots,\,a_{i-1},\, a_i,\,a_{i+1},\dots,\,a_n) = f(a_1,\dots,\,a_{i-1},\, a_i',\,a_{i+1},\dots,\,a_n)\,.
\end{equation*}
\end{definition}
Notice that this definition also applies to relations, since we view them as predicates, that is, mappings to $\{\top,\,\bot\}$. Thus, we may speak of relations with or without dummy variables.
\begin{convention}
Since it is often sufficient to consider relations up to an order of variables and/or up to dummy variables, we sometimes simplify the notation by not distinguishing between relations and the formulas that define them. For example, instead of writing:
\begin{equation*}
    \rho(x,\,y,\,z) \coloneqq (x=0\lor y+z=1)\,,\quad S\coloneqq \qpp{\rho}\,,
\end{equation*}
we use a more compact notation:
\begin{equation*}
    S\coloneqq\qpp{x=0\lor y+z=1}\,.
\end{equation*}
\end{convention}
\section{Canonical relations}\label{sec:canpreds}
We fix $k\in\N$ in this section. We define the following class of relations.
\begin{definition}\label{def:canonicalrelations}
A $k$-sorted relation without dummy variables is called \textit{canonical} if, up to an order of its variables, it may be expressed in one of the following ways:
\begin{enumerate}[label=\textup{(c\arabic*)},leftmargin=1.4\parindent]
    \item $x^i = 0 \lor y^i = 1$ \hspace*{\fill} for $i\in E_k$,
    \item $x^i = y^i \lor u^j = b$\hspace*{\fill} for $i,\,j\in E_k$, $b\in\{0,\,1\}$,
    \item $x^i = y^i \lor u^j = v^j$ and $x^i = y^i \lor y^i = z^i$ \hspace*{\fill} for $i,\,j\in E_k$, $i\neq j$,
    \item $x^i + y^i = 1$\hspace*{\fill} for $i\in E_k$,
    \item $x^i + y^i = u^j + v^j$\hspace*{\fill} for $i,\,j\in E_k$,
    \item $x^{s_1} + x^{s_2} + \dots + x^{s_n} = b$, where $n\geq 2$, $s_1,\dots,\,s_n\in E_k$ are pairwise distinct integers, and $b\in\{0,\,1\}$,
    \item \begin{equation*}
        \disj{\begin{array}{r@{\mskip\medmuskip}c@{\mskip\medmuskip}l}
        x_1^{s_1} = b_1 \lor x_2^{s_1} = b_1 \,\lor & \cdots & \lor\, x_{m_1}^{s_1} = b_1  \\
        x_1^{s_2} = b_2 \lor x_2^{s_2} = b_2 \,\lor & \cdots & \lor\, x_{m_2}^{s_2} = b_2 \\
        &\vdots& \\
        x_1^{s_n} = b_n \lor x_2^{s_n} = b_n \,\lor & \cdots & \lor\, x_{m_n}^{s_n} = b_n
    \end{array}}\,,
    \end{equation*}
    where $s_1,\dots,\,s_n\in E_k$ are pairwise distinct integers, and $b_1,\dots,\,b_n\in\{0,\,1\}$.
\end{enumerate}
By $\CR^k\subseteq \KR^k$, we denote the set of all $k$-sorted canonical relations.
\end{definition}
Theorem \ref{thm:KR} offers a~convenient way to describe relational clones on a~two-element domain. Analogously, the following theorems -- arguably the main results of this paper -- establish that canonical relations are sufficient to describe quantified relational clones and that they are, in some sense, the simplest relations with this property.
\begin{theorem}\label{thm:CR}
Let $S\subseteq R_{\{0,\,1\}}^k$. Then
\begin{equation*}
    \qpp{S} = \qpp{\qpp{S}\cap \CR^k}\,.
\end{equation*}
\end{theorem}
\begin{theorem}\label{thm:propCR}
If $\rho\in\CR^k$, then
\begin{enumerate}
    \item there is no $\rho'\in\qpp{\rho}$ of lower arity than $\rho$ such that $\qpp{\rho'} = \qpp{\rho}$;
    \item for all $n\in\N$ and $\sigma_1,\,\sigma_2,\dots,\,\sigma_n\in\qpp{\rho}$ such that
    \begin{equation*}
        \forall i\in E_n\quad \qpp{\sigma_i} \subsetneq \qpp{\rho}\,,
    \end{equation*}
    we have $\qpp{\sigma_1,\,\sigma_2,\dots,\,\sigma_n}\subsetneq\qpp{\rho}$\,.
\end{enumerate}
\end{theorem}
We divide the proof of Theorem \ref{thm:CR} into the following lemmata outlining the constructions. These lemmata are stated below without proofs and will be proved in Section \ref{sec:defproofscanon}, together with Theorem \ref{thm:propCR}.

\begin{lemma}\label{thm:gaussrearrangement}
For every relation $\rho'\in\KR^k$, there are constants $m,\,n,\,l\in\N_0$, and $p_i,\,q_j,\,r_h\in E_k$, $a_{i,j},\,b_i,\,c_h\in\{0,\,1\}$ for all $i\in E_m$, $j\in E_n$, $h\in E_l$ such that the relation $\rho$ defined by
\begin{equation*}
    \rho\!\left(\begin{array}{@{\,}l@{\,}}
        x_1^{p_1},\dots,\,x_m^{p_m}, \\
        y_1^{q_1},\dots,\,y_n^{q_n}, \\
        z_1^{r_1},\dots,\,z_l^{r_l}
    \end{array}\right) = \disj{\begin{array}{r@{\mskip\medmuskip}c@{\mskip\medmuskip}l}
        x_1^{p_1} & = & a_{1,1}y_1^{q_1} + \dots + a_{1,n}y_n^{q_n} + b_1 \\
        &\vdots&\\
        x_m^{p_m} & = & a_{m,1}y_1^{q_1} + \dots + a_{m,n}y_n^{q_n} + b_m\\
        z_1^{r_1} & = & c_1 \\
        & \vdots & \\
        z_l^{r_l} & = & c_l
    \end{array}}
\end{equation*}
is similar to $\rho'$.
\end{lemma}
\begin{lemma}\label{thm:decomposition}
Suppose
\begin{align*}
    \rho\!\left(\begin{array}{@{\,}l@{\,}}
        x_1^{p_1},\dots,\,x_m^{p_m}, \\
        y_1^{q_1},\dots,\,y_n^{q_n}, \\
        z_1^{r_1},\dots,\,z_l^{r_l}
    \end{array}\right) &= \disj{\begin{array}{r@{\mskip\medmuskip}c@{\mskip\medmuskip}l}
        x_1^{p_1} & = & a_{1,1}y_1^{q_1} + \dots + a_{1,n}y_n^{q_n} + b_1 \\
        &\vdots&\\
        x_m^{p_m} & = & a_{m,1}y_1^{q_1} + \dots + a_{m,n}y_n^{q_n} + b_m\\
        z_1^{r_1} & = & c_1 \\
        & \vdots & \\
        z_l^{r_l} & = & c_l
    \end{array}}\,,\\
    \sigma\!\left(\begin{array}{@{\,}l@{\,}}
        u_{1}^{p_1},\dots,\,u_{m}^{p_m}, \\
        v_{1}^{p_1},\dots,\,v_{m}^{p_m}, \\
        z_1^{r_1},\dots,\,z_l^{r_l}
    \end{array}\right)&= \disj{\begin{array}{r@{\mskip\medmuskip}c@{\mskip\medmuskip}l}
        u_{1}^{p_1} & = & v_{1}^{p_1} \\
        &\vdots&\\
        u_{m}^{p_m} & = & v_{m}^{p_m} \\
         z_1^{r_1} & = & c_1 \\
        & \vdots & \\
        z_l^{r_l} & = & c_l
    \end{array}}\,,\\
    \lambda_1(x_1^{p_1},\,y_1^{q_1},\dots,\,y_n^{q_n}) &= (x_1^{p_1} = a_{1,1}y_1^{q_1} + \dots + a_{1,n}y_n^{q_n} + b_1)\,,\\
    &\hspace*{0.5em}\vdots\\
    \lambda_m(x_m^{p_m},\,y_1^{q_1},\dots,\,y_n^{q_n}) &= (x_m^{p_m} = a_{m,1}y_1^{q_1} + \dots + a_{m,n}y_n^{q_n} + b_m)\,,
\end{align*}
where $m,\,n,\,l\in\N_0$, and $p_i,\,q_j,\,r_h\in E_k$, $a_{i,j},\,b_i,\,c_h\in\{0,\,1\}$ for all $i\in E_m$, $j\in E_n$, $h\in E_l$. Then
\begin{enumerate}
    \item $\sigma\in\pp{\rho}$,
    \item $\qpp{\rho} = \qpp{\sigma,\,\lambda_1,\dots,\,\lambda_m}$.
\end{enumerate}
\end{lemma}
\begin{lemma}\label{thm:canon45}
Let
\begin{multline*}
    \rho(x_1^1,\dots,\,x_{l_1}^1,\,x_1^2,\dots,\,x_{l_2}^2,\dots,\,x_1^k,\dots,\,x_{l_k}^k) =\\= (x_1^1 + \dots + x_{l_1}^1 + x_{1}^2 + \dots + x_{l_2}^2+\dots+x_{1}^k + \dots + x_{l_k}^k = b)\,,
\end{multline*}
where $l_1,\dots,\,l_k\in\N_0$ and $l_1 + \dots + l_k \geq 3$. We define
\begin{align*}
    I &= \{i\in E_k\mid l_i\neq 0\}\,, \qquad \text{(sorts with at least one variable)}\\
    O &= \{i\in E_k\mid l_i\text{ is odd}\}\,. \quad \text{(sorts with an odd number of variables)}
\end{align*}
Then:
\begin{enumerate}
    \item If $O = \emptyset$ and $b = 0$, then
    \begin{equation*}
        \qpp{\rho} = \qpp{\{x^i+y^i = u^j + v^j \mid i,\,j\in I\}}\,.
    \end{equation*}
    \item If $O = \emptyset$ and $b=1$, then
    \begin{equation*}
        \qpp{\rho} = \qpp{\{(x^i+y^i = u^j + v^j),\,(x^i + y^i = 1) \mid i,\,j\in I\}}\,.
    \end{equation*}
    \item If $O \neq \emptyset$, letting $\{p_1,\dots,\,p_m\} = O$, we have
    \begin{equation*}
        \qpp{\rho} = \qpp{\{x^{p_1} +\dots + x^{p_m} =  b\}\cup\{x^i+y^i = u^j + v^j \mid i,\,j\in I\}}\,.
    \end{equation*}
\end{enumerate}
\end{lemma}
\begin{lemma}\label{thm:canon2}
Let $p_1,\dots,\,p_m,\,r_1,\dots,\,r_l\in E_k$ and $c_1,\dots,\,c_l\in\{0,\,1\}$. Then
\begin{equation*}
    \qpp{\disj{\begin{array}{r@{\mskip\medmuskip}c@{\mskip\medmuskip}l}
        u_{1}^{p_1} & = & v_{1}^{p_1} \\
        &\vdots&\\
        u_{m}^{p_m} & = & v_{m}^{p_m} \\
        z_1^{r_1} & = & c_1 \\
        z_2^{r_2} & = & c_2 \\
        & \vdots & \\
        z_l^{r_l} & = & c_l
    \end{array}}} = \qpp{\disj{\begin{array}{r@{\mskip\medmuskip}c@{\mskip\medmuskip}l}
        u_{1}^{p_1} & = & v_{1}^{p_1} \\
        &\vdots&\\
        u_{m}^{p_m} & = & v_{m}^{p_m} \\
        z_2^{r_2} & = & c_2 \\
        & \vdots & \\
        z_{l}^{r_{l}} & = & c_{l}
    \end{array}},\,\disj{\begin{array}{r@{\mskip\medmuskip}c@{\mskip\medmuskip}l}
        u_{1}^{p_1} & = & v_{1}^{p_1} \\
        z_1^{r_1} & = & c_1
    \end{array}}}\,.
\end{equation*}
\end{lemma}
\begin{lemma}\label{thm:canon3}
Let $p_1,\dots,\,p_m\in E_k$. Then
\begin{equation*}
    \qpp{\disj{\begin{array}{r@{\mskip\medmuskip}c@{\mskip\medmuskip}l}
        u_{1}^{p_1} & = & v_{1}^{p_1} \\
        u_{2}^{p_2} & = & v_{2}^{p_2} \\
        &\vdots&\\
        u_{m}^{p_m} & = & v_{m}^{p_m} 
    \end{array}}} = \qpp{\disj{\begin{array}{r@{\mskip\medmuskip}c@{\mskip\medmuskip}l}
        u_{2}^{p_2} & = & v_{2}^{p_2} \\
        &\vdots&\\
        u_{m}^{p_m} & = & v_{m}^{p_m}
    \end{array}},\,\disj{\begin{array}{r@{\mskip\medmuskip}c@{\mskip\medmuskip}l}
        u_{1}^{p_1} & = & v_{1}^{p_1} \\
        u_{2}^{p_2} & = & v_{2}^{p_2} 
    \end{array}}}\,.
\end{equation*}
Furthermore, if $i\in E_k$, then
\begin{equation*}
    \qpp{x^i = y^i\lor u^i = v^i} = \qpp{x^i = y^i \lor y^i = z^i}\,.
\end{equation*}
\end{lemma}
\begin{lemma}\label{thm:disjeq01}
Let $i,\,r_1,\dots,\,r_l\in E_k$ and $c_1,\dots,\,c_l\in\{0,\,1\}$. Then
\begin{equation*}
    \qpp{\disj{\begin{array}{r@{\mskip\medmuskip}c@{\mskip\medmuskip}l}
        x^i & = & 0 \\
        y^i & = & 1 \\
        z_1^{r_1} & = & c_1 \\
        z_2^{r_2} & = & c_2 \\
        & \vdots & \\
        z_l^{r_l} & = & c_l
    \end{array}}} = \qpp{\disj{\begin{array}{r@{\mskip\medmuskip}c@{\mskip\medmuskip}l}
        x^i & = & 0 \\
        y^i & = & 1 \\
        z_2^{r_2} & = & c_2 \\
        & \vdots & \\
        z_l^{r_l} & = & c_l
    \end{array}},\,\disj{\begin{array}{r@{\mskip\medmuskip}c@{\mskip\medmuskip}l}
        x^i & = & 0 \\
        y^i & = & 1 \\
        z_1^{r_1} & = & c_1 \\
    \end{array}}}\,.
\end{equation*}
\end{lemma}
\begin{lemma}\label{thm:canon12}
Let $i,\,j\in E_k$ and $b\in\{0,\,1\}$. Then
\begin{equation*}
    \qpp{\disj{\begin{array}{r@{\mskip\medmuskip}c@{\mskip\medmuskip}l}
        x^i & = & 0 \\
        y^i & = & 1 \\
        z^j & = & b \\
    \end{array}}} = \qpp{\disj{\begin{array}{r@{\mskip\medmuskip}c@{\mskip\medmuskip}l}
        x^i & = & y^i \\
        z^j & = & b \\
    \end{array}}, \disj{\begin{array}{r@{\mskip\medmuskip}c@{\mskip\medmuskip}l}
        x^i & = & 0 \\
        y^i & = & 1 \\
    \end{array}}}\,.
\end{equation*}
\end{lemma}
\begin{proof}[Proof of Theorem \ref{thm:CR}]
We will refer to the canonical relations by their numbering from Definition \ref{def:canonicalrelations}. The proof is a straightforward application of lemmata \ref{thm:gaussrearrangement}, \ref{thm:decomposition}, \ref{thm:canon45}, \ref{thm:canon2}, \ref{thm:canon3}, \ref{thm:disjeq01}, and \ref{thm:canon12}.

By Theorem~\ref{thm:KR}, it suffices to consider only the relations in $\KR^k$. For any $\rho'\in\KR^k$, we consider the similar relation $\rho\in\KR^k$ given by Lemma \ref{thm:gaussrearrangement}, for which we consider the relations $\sigma,\,\lambda_1,\dots,\,\lambda_m\in\KR^k$ given by Lemma \ref{thm:decomposition}. We have $\qpp{\rho'} = \qpp{\sigma,\,\lambda_1,\dots,\,\lambda_m}$.
\begin{itemize}
    \item By Lemma \ref{thm:canon45}, the relations $\lambda_1,\dots,\,\lambda_m$ can be described by some canonical relations (c4), (c5), (c6), and/or unary relations of type~(c7).
    \item The relation $\sigma$ is either similar to some canonical relation (c1) or (c7), or it can be described by some canonical relations (c2), (c3), and/or (c1) by the lemmata \ref{thm:canon2}, \ref{thm:canon3}, \ref{thm:disjeq01}, and/or \ref{thm:canon12}.\qedhere
\end{itemize}
\end{proof}



\section{Closedness criteria and Post's lattice theorem}\label{sec:closednessandpost}
Although all quantified relational clones are uniquely described by the canonical relations they contain, two different sets of canonical relations do not necessarily generate different quantified relational clones. In this section, we introduce elementary operations, which are less cumbersome to work with than general qpp formulas, and we describe the properties that make them a~valuable tool to decide whether two sets $S,\,S'\subseteq \CR^k$ generate the same quantified relational clone or not. Using these tools, we subsequently describe all $1$-sorted quantified relational clones, which can be extended to a classification of relational clones, providing a concise proof of Post's lattice theorem.
\subsection{Elementary operations}\label{sec:eo}
\begin{definition}[Elementary operations]\label{def:eo}
Let $k\in\N$. We define the~following \textit{elementary operations} on $R_{\{0,\,1\}}^k$.
\begin{enumerate}[label=\textup{(eo\arabic*)},leftmargin=1.4\parindent]
    \item \textit{Appending} or \textit{removing} a~dummy variable.
    \item \textit{Permutation} of variables. For a~relation $\rho$ of an arity $n$ and a~permutation $\pi$ on $\{1,\dots,\,n\}$ we define a~new relation $\rho_\pi$ as
    \begin{equation*}
        \rho_\pi(x_1,\dots,\,x_n) = \rho(x_{\pi(1)},\dots,\,x_{\pi(n)})\,.
    \end{equation*}
    \item \textit{Identification} of variables. For a~relation $\rho$ whose first two variables have the same sort, we define a~new relation $\rho'$ as
    \begin{equation*}
        \rho'(x_1,x_2,\dots,\,x_{n-1}) = \rho(x_1,\,x_1,\,x_2,\dots,\,x_{n-1})\,.
    \end{equation*}
    \item \textit{Composition} of  relations. Let $\rho_1$ and $\rho_2$ be relations whose first variables are non-dummy and have the same sort. The new relation $\rho'$ is defined by
    \begin{multline*}
        \rho'(x_1,\dots,\,x_{n-1},\,y_1,\dots,\,y_{m-1}) =\\= \exists z\, \rho_1(z,\,x_1,\dots,\,x_{n-1}) \land \rho_2(z,\,y_1,\dots,\,y_{m-1})\,.
    \end{multline*}
    \item \textit{Universal quantification} of a variable. For a~relation $\rho$ we define a~new relation $\rho'$ as
    \begin{equation*}
        \rho'(x_1,\dots,\,x_{n-1}) = \forall y\, \rho(y,\,x_1,\,x_2,\dots,\,x_{n-1})\,.
    \end{equation*}
    \item \textit{Conjunction} of two relations. Let $\rho_1$~and~$\rho_2$ be $n$-ary relations such that the $i$-th variable of $\rho_1$ has the same sort as the $i$-th variable of $\rho_2$ for every $i\in E_n$. The new relation $\rho'$ is defined by
    \begin{equation*}
        \rho'(x_1,\dots,\,x_n) = \rho_1(x_1,\dots,\,x_n) \land \rho_2(x_1,\dots,\,x_n)\,.
    \end{equation*}
\end{enumerate}
For $i\in E_6$ and $S\subseteq R_{\{0,\,1\}}^k$, we define $\eo{S}{i}$ as the set of all relations in $R_{\{0,\,1\}}^k$ that can be obtained from the elements of $S\cup\{\sigma_\bot,\,\sigma_{=}^1,\dots,\,\sigma_{=}^k\}$ by a sequence of elementary operations (eo1)--(eo$i$), and $\eoc{S}$ as the set of all relations in $R_{\{0,\,1\}}^k$ that can be obtained by a sequence of elementary operations (eo6).
\end{definition}
Let us now describe the properties of elementary operations. Informally, we claim that elementary operations are equivalent to general qpp formulas and that any relation in $\qpp{S}$ (for some set $S\subseteq\KR^k$), which can not be constructed using the elementary operations (eo1)--(eo5), is simply a conjunction of some relations that can be constructed in such way (Lemma~\ref{thm:eo} and Theorem~\ref{thm:eoconj}). Moreover, we claim that the application of elementary operations (eo1)--(eo5) on relations in $\KR^k$ can be computed explicitly (Remark~\ref{rem:applyingeo} and Observation~\ref{thm:KRclosedoneo5}) implying that characterization of quantified relational clones is, to some extent, purely a mechanical task (Remark~\ref{rem:closedness}).
\begin{lemma}\label{thm:eo}
Let $k\in\N$.
\begin{enumerate}
    \item Assume a relation $\rho\in R_{\{0,\,1\}}^k$ is defined by
    \begin{equation*}
        \rho(x_1,\dots,\,x_n) = \exists y\, \rho_1(v_{1,1},\dots,\,v_{1,n_1})\land\dots\land\rho_m(v_{m,1},\dots,\,v_{m,n_m})
    \end{equation*}
    where $\rho_1,\dots,\,\rho_m\in R_{\{0,\,1\}}^k$ and $v_{i,j}\in\{x_1,\dots,\,x_n,\,y\}$ are symbols for variables. There are relations $\sigma_1,\dots,\,\sigma_l \in \eo{\rho_1,\dots,\,\rho_m}{4}$ such that
    \begin{equation*}
        \rho(x_1,\dots,\,x_n) =  \sigma_1(x_1,\dots,\,x_n)\land\dots\land\sigma_l(x_1,\dots,\,x_n)\,.
    \end{equation*}
    \item Let $S\subseteq R_{\{0,\,1\}}^k$. For every $\rho\in\qpp{S}$, there are relations $\sigma_1,\dots,\,\sigma_l \in \eo{S}{5}$ such that
    \begin{equation*}
        \rho(x_1,\dots,\,x_n) =  \sigma_1(x_1,\dots,\,x_n)\land\dots\land\sigma_l(x_1,\dots,\,x_n)\,.
    \end{equation*}
\end{enumerate}
\end{lemma}
\begin{proof}
It is straightforward to see that there are relations $\rho_1',\dots,\,\rho_m'$ defined from relations $\rho_1,\dots,\,\rho_m$ by rearranging, permuting variables, identifying variables, and appending/removing dummy variables, such that
\begin{multline*}
    \rho(x_1,\dots,\,x_n) = \exists y\, \rho_1'(y,\,x_1,\dots,\,x_n)\land\dots\land\rho_p'(y,\,x_1,\dots,\,x_n)\,\land\\
    \land \rho_{p+1}'(x_1,\dots,\,x_n)\land\dots\land\rho_m'(x_1,\dots,\,x_n)\,,
\end{multline*}
where the first variable of relations $\rho_1,\dots,\,\rho_p$ is non-dummy. For the rest of the proof, we will write $\boldsymbol{x}$ instead of $x_1,\dots,\,x_n$. 

Noticing that the relations $\rho_{p+1}',\dots,\,\rho_m'$ do not depend on the variable $y$, we can rearrange the original expression:
\begin{equation}
    \rho(\boldsymbol{x}) = \rho_{p+1}'(\boldsymbol{x})\land\dots\land\rho_m'(\boldsymbol{x})\, \land\exists y\, \rho_1'(y,\,\boldsymbol{x})\land\dots\land\rho_p'(y,\,\boldsymbol{x})\,. \label{eq:eoexistentialqelimination}
\end{equation}
We now assume that $p\geq 2$, as the case $p=0$ is trivial, and the case $p=1$ can be interpreted as a composition of $\rho_1'$ with itself and subsequent identifications of variables. The following claim captures the core idea of the proof.
\begin{claim}\label{thm:rewriteascompositions}
Let us denote
\begin{align*}
    \tau(\boldsymbol{x}) &= \exists y\, \rho_1'(y,\,\boldsymbol{x})\land\dots\land\rho_p'(y,\,\boldsymbol{x})\,,\\
    \tau'(\boldsymbol{x}) &= \bigwedge_{\substack{i,\,j\,\in\,E_p\\ i\, <\, j}} \exists z_{i,j}\, \rho_i'(z_{i,j},\,\boldsymbol{x})\land\rho_j'(z_{i,j},\,\boldsymbol{x})\,.
\end{align*}
Then $\tau(\boldsymbol{x}) = \tau'(\boldsymbol{x})$.
\end{claim}
\begin{proof}[Proof of Claim \ref{thm:rewriteascompositions}]
Fix $\boldsymbol{x} = \boldsymbol{a}\in\{0,\,1\}^n$. It is enough to prove that $\tau(\boldsymbol{a}) = \top$ if and only if $\tau'(\boldsymbol{a}) = \top$.

Assume $\tau(\boldsymbol{a}) = \top$, any witness for the existential quantifier in $\tau(\boldsymbol{a})$ is then a witness for all existential quantifiers in $\tau'(\boldsymbol{a})$, implying $\tau'(\boldsymbol{a}) = \top$.

Conversely, assume $\tau'(\boldsymbol{a}) = \top$ and suppose $0$ is not a witness for $y$ in $\tau(\boldsymbol{a})$. This implies that $\rho_q'(0,\,\boldsymbol{a}) = \bot$ for some $q\in E_p$. Since $\tau'(\boldsymbol{a}) = \top$, we especially know that
\begin{equation*}
    \exists z_{q,j}\, \rho_q'(z_{q,j},\,\boldsymbol{a})\land\rho_j'(z_{q,j},\,\boldsymbol{a}) = \top
\end{equation*}
for every $j\in E_p\setminus\{q\}$. This must be witnessed by $z_{q,j} = 1$, therefore $y=1$ is a witness for the existential quantifier in $\tau(\boldsymbol{a})$ and thus $\tau(\boldsymbol{a}) = \top$.
\end{proof}
Using Claim \ref{thm:rewriteascompositions}, the formula \eqref{eq:eoexistentialqelimination} simplifies to
\begin{equation*}
    \rho(\boldsymbol{x}) =\rho_{p+1}'(\boldsymbol{x})\land\dots\land\rho_m'(\boldsymbol{x})\,\land \bigwedge_{\substack{i,\,j\,\in\,E_p\\ i\, <\, j}} \underbrace{\exists z_{i,j}\, \rho_i'(z_{i,j},\,\boldsymbol{x})\land\rho_j'(z_{i,j},\,\boldsymbol{x})}_{\text{composition of $\rho_i'$ and $\rho_j'$}}\,.
\end{equation*}
This concludes the proof of statement (1) of the theorem.

Finally, statement (2) follows from statement (1) and from the distributivity of universal quantifiers over conjunctions -- specifically, by inductively eliminating all quantifiers from a general qpp formula.
\end{proof}
\begin{theorem}\label{thm:eoconj}
Let $k\in\N$. For every $S\subseteq R_{\{0,\,1\}}^k$, we have that
    \begin{equation*}
        \qpp{S} = \eo{S}{6} = \eoc{\eo{S}{5}}\,.
    \end{equation*}
\end{theorem}
\begin{proof}
This follows directly from statement (2) of Lemma \ref{thm:eo}, combined with the straightforward observation that all elementary operations are qpp formulas.
\end{proof}
\begin{remark}
Lemma \ref{thm:eo} may be extended to relations over an arbitrary domain of size $n\in\N$, where, instead of composition of $2$ relations, we similarly define a composition of $n$ relations. We omit the details as they are beyond the scope of this paper.
\end{remark}
Based on Theorem \ref{thm:eoconj}, we summarize the methods of using the elementary operations to describe (quantified) relational clones in the following remark.
\begin{remark}[Applying elementary operations]\label{rem:applyingeo}
Let $k\in\N$. Assume $\rho,\,\rho'\in\KR^k$ are relations of the form
\begin{align*}
    \rho(x,\,y_1,\dots,\,y_n) &= (a_1 y_1+\dots + a_n y_n + b = x) \lor \sigma(y_1,\dots,\,y_n)\,,\\
    \rho'(x,\,z_1,\dots,\,z_{m}) &= (a_1' z_1+\dots + a_{m}' z_m + b' = x) \lor \sigma'(z_1,\dots,\,z_{m})\,,
\end{align*}
for some $\sigma,\,\sigma'\in\KR^k$ and $a_1,\dots,\,a_n,\,a_1',\dots,\,a_{m}',\,b,\,b'\in\{0,\,1\}$. Let us observe that
\begin{equation*}
    \forall x\, \rho(x,\,y_1,\dots,\,y_n) = \sigma(y_1,\dots,\,y_n)\,,
\end{equation*}
and
\begin{multline*}
    \exists x\, \rho(x,\,y_1,\dots,\,y_n) \land \rho'(x,\,z_1,\dots,\,z_m) =\\= (a_1 y_1+\dots + a_n y_n + a_1'z_1+\dots+a_m'z_m + b + b'=0)\mskip\medmuskip \lor\\ \lor \sigma(y_1,\dots,\,y_n) \lor \sigma'(z_1,\dots,\,z_{m})\,,
\end{multline*}
where addition is taken modulo $2$. Colloquially, we see that (eo4) only ``joins'' some equations together and (eo5) ``removes'' equations.

From this example and based on the fact that applying (eo1)--(eo3) is simple, we may conclude that applying the elementary operations (eo1)--(eo5) on relations in $\KR^k$ is not only straightforward, but it always results in a~relation in $\KR^k$ (as we formally state in Observation \ref{thm:KRclosedoneo5}) whose representation as a disjunction of linear equations can be obtained directly from the representation of the original relations.
\end{remark}
\begin{observation}\label{thm:KRclosedoneo5}
Let $k\in\N$ and let $S\subseteq\KR^k$. Then $\eo{S}{5}\subseteq \KR^k$.
\end{observation}
As we have observed, working with each elementary operation, apart from (eo6), is straightforward and mechanical. While the computation of (eo6) is also mechanical in any given case, determining whether the resulting relation belongs to $\CR^k$ is equivalent to checking whether a union of affine spaces over ${0,\,1}$ is itself an affine space (an immediate consequence of De~Morgan's laws) and obtaining an explicit representation of such relation can be cumbersome. Nevertheless, as the following theorem suggests, in each case, only finitely many conjunctions need to be considered.
\begin{theorem}\label{thm:expressedbyconjunctions}
Let $k\in\N$ and $S\subseteq R_{\{0,\,1\}}^k$. Assume $\rho$ is a canonical relation such that $\rho\in\qpp{S}$ and $\rho\notin \eo{S}{5}$. Then $\rho$ is not a canonical relation of type (c7). In particular, there are only finitely many relations $\rho$ such that $\rho\in\qpp{S}$ and $\rho\notin\eo{S}{5}$.
\end{theorem}
\begin{proof}
By Theorem~\ref{thm:eoconj}, we may assume that
\begin{equation*}
    \rho(x_1,\dots,\,x_n) = \sigma_1(x_1,\dots,\,x_n)\land\dots\land\sigma_m(x_1,\dots,\,x_n)
\end{equation*}
for some $\sigma_1,\dots,\,\sigma_m\in\eo{S}{5}$. Suppose that $\rho$ is a canonical relation of type (c7) of arity $n$. For any such relation, there exists a~uniquely determined tuple $(a_1',\dots,\,a_n')\in\{0,\,1\}^n$ such that
\begin{equation*}
    \rho = A^n\setminus\{(a_1',\dots,\,a_n')\}\,.
\end{equation*}
This implies that for every $i\in E_m$, the relation $\sigma_i$ is either the full relation (i.e., the corresponding predicate is the constant $\top$ predicate) or equal to $\rho$. In either case, this leads to contradiction.

The second claim follows from the fact that for every $k\in \N$ there are finitely many canonical relations of types (c1)--(c6).
\end{proof}

We summarize the main results in the following remark.

\begin{remark}[Checking for closedness]\label{rem:closedness}
Suppose that we need to decide if a~set $S\subseteq\CR^k$ describes a relational clone, i.e. whether $S = \qpp{S}\cap\CR^k$. 

By Theorem~\ref{thm:eoconj} and Theorem~\ref{thm:expressedbyconjunctions}, this task reduces to computing $\eo{S}{5}$, which is a purely mechanical task by Remark~\ref{rem:applyingeo}, and subsequent checking of the following two conditions:
\begin{enumerate}
    \item $\eo{S}{5} \cap \CR^k = S$.
    \item For each of the finitely many canonical relations $\rho\notin S$ of types (c1)--(c6), we have
    \begin{equation*}
        \rho(x_1,\dots,\,x_n) \neq \bigwedge_{\substack{\sigma\,\in\,\eo{S}{5}\\ \rho\,\subseteq\,\sigma}} \sigma(x_1,\dots,\,x_n)\,.
    \end{equation*}
\end{enumerate}
If both of these conditions are satisfied, then $S = \qpp{S}\cap\CR^k$.

In conclusion, the entire process of describing qpp-closed sets and multi-sorted quantified relational clones on $\{0,\,1\}$ is, although nontrivial, mechanical in nature. To some extent, the same can be said about ``non-quantified'' relational clones. We explore this in Subsection~\ref{sec:post} (mainly Remark~\ref{rem:post}) by outlining the proof of Post's lattice theorem, and in Section~\ref{sec:lattice} (mainly Theorem~\ref{thm:indstr}), where we describe an inductive structure of the lattice of these clones.
\end{remark}

\subsection{1-sorted case and Post's lattice theorem}\label{sec:post}
As a demonstration of the practicality of the developed tools, we derive the classification of all $1$-sorted quantified relational clones. While this classification is well known and can be obtained via Post's lattice, the approach presented here provides a direct way to obtain it and, in turn, enables a direct proof of Post's lattice itself. Since the proof follows naturally from the previously developed tools and reduces to a routine verification, requiring no additional ideas, we omit the details and instead summarize the classification.
\begin{theorem}\label{thm:qpp_post}
Figure~\ref{fig:qpp_post} shows all the $1$-sorted quantified relational clones on~$\{0,\,1\}$ ordered by inclusion.
\end{theorem}
\begin{figure}[ht]
    \centering
    \begin{tikzcd}[scale cd=.57, column sep=0.5em, row sep=1.5em, ampersand replacement=\&, cramped, crossing over clearance=.75ex]
	\&\&\&\& \bullet \\
	\&\& {\disj{\begin{array}{r@{\mskip\medmuskip}c@{\mskip\medmuskip}l}
        x &=& y\\
        u &=& 0
    \end{array}}} \&\&\&\& {\disj{\begin{array}{r@{\mskip\medmuskip}c@{\mskip\medmuskip}l}
        x &=& y\\
        u &=& 1
    \end{array}}} \\
	\&\&\&\& {\disj{\begin{array}{r@{\mskip\medmuskip}c@{\mskip\medmuskip}l}
        x &=& y\\
        y &=& z
    \end{array}}} \\
	\bullet \&\&\&\& \bullet \&\&\&\& \bullet \\
	\bullet \&\&\& \bullet \& \bullet \& \bullet \&\&\& \bullet \\
	\begin{array}{c} \bigvee_{i=1}^n x_i = 0\\ \text{for every } n\in\N \end{array} \& \bullet \&\&\& {x+y=u+v} \&\&\& \bullet \& \begin{array}{c} \bigvee_{i=1}^n x_i = 1\\ \text{for every } n\in\N \end{array} \\
	\& \bullet \& \bullet \&\& \bullet \&\& \bullet \& \bullet \\
	\& {\bigvee_{i=1}^4 x_i = 0} \& \bullet \& \bullet \& \bullet \& \bullet \& \bullet \& {\bigvee_{i=1}^4 x_i = 1} \\
	\&\& {\bigvee_{i=1}^3 x_i = 0} \& \bullet \& {\disj{\begin{array}{r@{\mskip\medmuskip}c@{\mskip\medmuskip}l}
        x_1 &=& 0\\
        x_2 &=& 1
    \end{array}}} \& \bullet \& {\bigvee_{i=1}^3 x_i = 1} \\
	\&\& {\disj{\begin{array}{r@{\mskip\medmuskip}c@{\mskip\medmuskip}l}
        x_1 &=& 0\\
        x_2 &=& 0
    \end{array}}} \&\& \bullet \&\& {\disj{\begin{array}{r@{\mskip\medmuskip}c@{\mskip\medmuskip}l}
        x_1 &=& 1\\
        x_2 &=& 1
    \end{array}}} \\
	\&\&\& {x=0} \& {x+y=1} \& {x=1} \\
	\&\&\&\& {x+y=0}
	\arrow[curve={height=-30pt}, no head, from=4-5, to=1-5]
	\arrow[curve={height=9pt}, no head, from=5-5, to=3-5]
        \arrow[curve={height=-21pt}, no head, from=10-5, to=8-5]
        \arrow[curve={height=-18pt}, no head, from=11-4, to=5-4]
	\arrow[curve={height=-54pt}, no head, from=11-5, to=5-5]
	\arrow[curve={height=24pt}, no head, from=11-5, to=8-5]
	\arrow[curve={height=18pt}, no head, from=11-6, to=5-6]
	\arrow[curve={height=36pt}, no head, from=12-5, to=6-5]
	\arrow[curve={height=30pt}, no head, from=4-5, to=8-5, crossing over]
	\arrow[curve={height=60pt}, no head, from=7-5, to=1-5, crossing over]
	\arrow[no head, from=2-3, to=1-5]
	\arrow[no head, from=2-7, to=1-5]
	\arrow[no head, from=3-5, to=1-5]
	\arrow[no head, from=4-1, to=2-3]
	\arrow[no head, from=4-9, to=2-7]
	\arrow[no head, from=5-1, to=4-1]
	\arrow[no head, from=5-4, to=4-5]
	\arrow[no head, from=5-5, to=4-5]
	\arrow[no head, from=5-6, to=4-5, crossing over]
	\arrow[no head, from=5-9, to=4-9]
	\arrow[no head, from=6-1, to=5-1]
	\arrow[dashed, no head, from=6-2, to=4-1]
	\arrow[no head, from=6-5, to=5-4, crossing over]
	\arrow[no head, from=6-5, to=5-5]
	\arrow[no head, from=6-5, to=5-6]
	\arrow[dashed, no head, from=6-8, to=4-9]
	\arrow[no head, from=6-9, to=5-9]
	\arrow[dashed, no head, from=7-2, to=5-1]
	\arrow[no head, from=7-2, to=6-2]
	\arrow[no head, from=7-3, to=6-2]
	\arrow[no head, from=7-7, to=6-8]
	\arrow[dashed, no head, from=7-8, to=5-9]
	\arrow[no head, from=7-8, to=6-8]
	\arrow[dashed, no head, from=8-2, to=6-1]
	\arrow[no head, from=8-2, to=7-2]
	\arrow[no head, from=8-3, to=7-2]
	\arrow[no head, from=8-3, to=7-3]
	\arrow[no head, from=8-4, to=7-3, crossing over]
	\arrow[no head, from=8-4, to=7-5, crossing over]
	\arrow[no head, from=8-5, to=7-5]
	\arrow[no head, from=8-6, to=7-5, crossing over]
	\arrow[no head, from=8-6, to=7-7, crossing over]
	\arrow[no head, from=8-7, to=7-7]
	\arrow[no head, from=8-7, to=7-8]
	\arrow[dashed, no head, from=8-8, to=6-9]
	\arrow[no head, from=8-8, to=7-8]
	\arrow[no head, from=9-3, to=8-2]
	\arrow[no head, from=9-3, to=8-3]
	\arrow[no head, from=9-4, to=8-3, crossing over]
	\arrow[no head, from=9-4, to=8-4]
	\arrow[no head, from=9-5, to=8-4, crossing over]
	\arrow[no head, from=9-5, to=8-6, crossing over]
	\arrow[no head, from=9-6, to=8-6]
	\arrow[no head, from=9-6, to=8-7, crossing over]
	\arrow[no head, from=9-7, to=8-7]
	\arrow[no head, from=9-7, to=8-8]
	\arrow[no head, from=10-3, to=9-3]
	\arrow[no head, from=10-3, to=9-4, crossing over]
	\arrow[no head, from=10-5, to=9-4]
	\arrow[no head, from=10-5, to=9-5]
	\arrow[no head, from=10-5, to=9-6, crossing over]
	\arrow[no head, from=10-7, to=9-6, crossing over]
	\arrow[no head, from=10-7, to=9-7]
	\arrow[no head, from=11-4, to=10-3]
	\arrow[no head, from=11-4, to=10-5, crossing over]
	\arrow[no head, from=11-6, to=10-5, crossing over]
	\arrow[no head, from=11-6, to=10-7]
	\arrow[no head, from=12-5, to=11-4]
	\arrow[no head, from=12-5, to=11-5]
	\arrow[no head, from=12-5, to=11-6]
\end{tikzcd}

    \caption{Hasse diagram of all 1-sorted quantified relational clones on $\{0,\,1\}$. Each vertex denotes a quantified relational clone generated by all the relations below it.}
    \label{fig:qpp_post}
\end{figure}
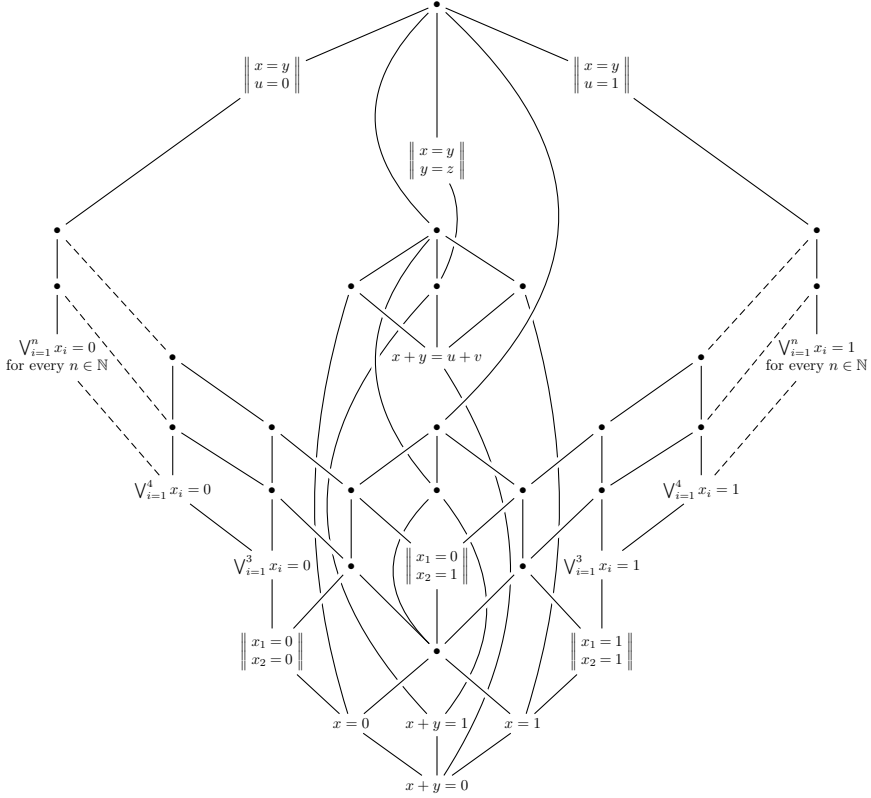
\begin{proof}[Proof (outline)]
It suffices to:
\begin{itemize}
    \item Verify that the relational clones generated by $x=0$, $x=1$, $x+y=1$, and $x+y=u+v$ are minimal by checking that their qpp closure does not contain any other canonical relation.
    \item Verify all the presented inclusions $\rho\in\qpp{\sigma}$ for $\rho,\,\sigma\in\CR^1$.
    \item Compute all joins in the lattice and check for closedness.
\end{itemize}
Each of these steps follows directly from the tools introduced in the previous section, mainly in Remark~\ref{rem:applyingeo} and Remark~\ref{rem:closedness}. For some details, we also point to the description of all quantified relational clones generated by a single canonical relation, as given in Table~\ref{tab:qppcanon} and Theorem~\ref{thm:qppcanon}.
\end{proof}

By the $\sPol$-$\Inv$ Galois connection (Theorem \ref{thm:Galois}), quantified relational clones correspond to surjective clones. Thus, Theorem~\ref{thm:qpp_post} also implicitly provides the classification of surjective clones.

Since the only non-surjective operations on $\{0,\,1\}$, up to arity, are the two constant operations, the classification of all clones -- Post's lattice theorem -- follows as a direct corollary of Theorem~\ref{thm:qpp_post}. We formalize this in the following remark.
\begin{remark}[Post's lattice theorem \cite{Post1920,Post1941}] \label{rem:post}
Let $\mathfrak{0}$ and $\mathfrak{1}$ denote the unary constant operations returining $0$ and $1$, respectively.

It follows directly from the definition that every clone is obtained by adding none, one, or both the constant operation to each surjective clone. That is, for each quantified relational clone $S$ it is enough to consider the following (not necessarily distinct) clones
\begin{equation*}
    \Pol S\,,\quad \Clo(\Pol S\cup\{\mathfrak{0}\})\,,\quad \Clo(\Pol S\cup\{\mathfrak{1}\})\,,\quad \Clo(\Pol S\cup\{\mathfrak{0},\,\mathfrak{1}\})\,.
\end{equation*}

Moreover, this classification can be derived without explicit clone computations. By the well-known $\Pol$-$\Inv$ Galois connection (see \cite{Geiger1968,Bodnarchuk1969I,Bodnarchuk1969II} or Corollary~\ref{thm:invpol}), we obtain
\begin{align*}
    \Inv\Pol S &= S\,,\\
    \Inv\Clo(\Pol S\cup\{\mathfrak{0}\}) &= S\cap \Inv\{\mathfrak{0}\}\,,\\
    \Inv\Clo(\Pol S\cup\{\mathfrak{1}\}) &= S\cap \Inv\{\mathfrak{1}\}\,,\\
    \Inv\Clo(\Pol S\cup\{\mathfrak{0},\,\mathfrak{1}\}) &= S\cap \Inv\{\mathfrak{0},\,\mathfrak{1}\}\,.
\end{align*}
This follows from the fact that in the $\Pol$-$\Inv$ connection, joins in the lattice of clones correspond to meets in the lattice of relational clones. 

Thus, to classify all clones, it suffices to compute
\begin{equation*}
    S\,,\quad S\cap \Inv\{\mathfrak{0}\}\,,\quad S\cap \Inv\{\mathfrak{1}\}\,,\quad S\cap \Inv\{\mathfrak{0},\,\mathfrak{1}\}\,.
\end{equation*}

To simplify this even further, suppose $S = \qpp{T}$ for some $T\subseteq \CR^1$. Then we claim:
\begin{claim}
\begin{align*}
    S\cap\Inv\{\mathfrak{0}\} &= \eoc{\eo{T}{5}\cap\Inv\{\mathfrak{0}\}}\,,\\
    S\cap\Inv\{\mathfrak{1}\} &= \eoc{\eo{T}{5}\cap\Inv\{\mathfrak{1}\}}\,,\\
    S\cap\Inv\{\mathfrak{0},\,\mathfrak{1}\} &= \eoc{\eo{T}{5}\cap\Inv\{\mathfrak{0},\,\mathfrak{1}\}}\,.
\end{align*}
\end{claim}
\begin{proof}
We prove the first equality; the remaining cases are analogous. Theorem~\ref{thm:eoconj} gives the following equality:
\begin{equation*}
    S\cap\Inv\{\mathfrak{0}\} = \eoc{\eo{T}{5}}\cap\Inv\{\mathfrak{0}\}\,.
\end{equation*}
Hence, it is enough to observe that
\begin{equation*}
    \eoc{\eo{T}{5}}\cap\Inv\{\mathfrak{0}\} = \eoc{\eo{T}{5}\cap\Inv\{\mathfrak{0}\}}\,.
\end{equation*}
This holds because a relation defined as a conjunction of other relations is preserved by $\mathfrak{0}$ if and only if all the conjuncts are preserved by $\mathfrak{0}$.
\end{proof}

Thus, to obtain the full classification of $1$-sorted relational clones on $\{0,\,1\}$, i.e. the relational counterpart of the Post's lattice, it suffices to:
\begin{itemize}
    \item Compute $\eo{T}{5}$ for all sets $T\subseteq \CR^1$ described by Theorem~\ref{thm:qpp_post}.
    \item Compute the intersections of $\eo{T}{5}$ with $\Inv\{\mathfrak{0}\}$, $\Inv\{\mathfrak{1}\}$, and $\Inv\{\mathfrak{0},\,\mathfrak{1}\}$, which amounts to selecting relations preserved by these constant operations.
\end{itemize}
This yields the classification of $1$-sorted relational clones, which correspond to $1$-clones via the $\Pol$-$\Inv$ Galois connection, thereby providing a concise and straightforward proof of Post's lattice theorem \cite{Post1920,Post1941}.
\end{remark}

\section{Lattice of multi-sorted clones}\label{sec:lattice}
In this section, we use the description of canonical relations to study the properties of the lattice of $k$-sorted quantified relational clones. The main result, Theorem \ref{thm:embedding}, particularly states that the lattice in question has quite a simple structure and admits an embedding into the product of a finite poset with the poset of downsets in $\N_0^{2k}$. Noticing the inductive structure of a full lattice of $k$-clones in Subsection \ref{sec:inductivestructure}, we then get a short proof of Taimanov's theorem \cite{Taimanov1983,Taimanov1984thesis,Taimanov2019,Taimanov2020,Taimanov2022} about the finiteness of generating set of every $k$-clone.

In this section, all observations follow directly from the definitions and are therefore presented without proofs. Furthermore, the proofs of all lemmata are deferred to Subsection \ref{sec:defproofslattice} due to their technical nature.
\subsection{Embedding of the lattice}
We assume the reader is familiar with the notion of partially ordered sets (shortly, posets) and begin by defining some perhaps less established concepts.
\begin{definition}[Closure operator]
Let $A$ be a~nonempty set. An~\textit{algebraic closure operator} on $A$ is a map $\cls: 2^A\to 2^A$ with the following properties
\begin{enumerate}[label=\textup{(cl\arabic*)},leftmargin=1.4\parindent]
    \item $B\subseteq \cls B$\hspace*{\fill} for all $B\subseteq A$,
    \item $B\subseteq C\implies \cls B\subseteq \cls C$ \hspace*{\fill} for all $B,\,C\subseteq A$,
    \item $\cls\cls B = \cls B$ \hspace*{\fill} for all $B\subseteq A$,
    \item $\cls B = \bigcup_{C\,\subseteq\, B,\, |C|\,<\,\infty} \cls C$\hspace*{\fill} for all $B\subseteq A$.
\end{enumerate}
A~set $B\subseteq A$ is called \textit{closed} with respect to $\cls$ if $B = \cls B$. We say a closed set $B$ is \textit{finitely generated} if there is a finite set $C\subseteq A$ such that $B = \cls C$.

By $\mathfrak{L}_{\cls}(A)$ we denote the set $\{B\subseteq A\mid B = \cls B\}$ of all closed subsets of $A$.
\end{definition}
Although, in general, closure operators are not required to satisfy the condition~(cl4), we will only be concerned with algebraic closure operators. We may also notice that the set $\mathfrak{L}_{\cls}(A)$ ordered by inclusion has a structure of a lattice, which, in particular, means that all the multi-sorted clones (or quantified relational clones, etc.) ordered by inclusion form a lattice; see Observation~\ref{thm:closureoperators}. However, we do not assume nor require this property in any of the following claims.
\begin{definition}[Order theory concepts I]
Let $(P,\,\leq)$ be a~poset. Any set $D\subseteq P$ satisfying
\begin{equation*}
    \forall\,q\in D\,\forall\,p\in P: p\leq q\implies p\in D
\end{equation*}
is called a~\textit{downset}; dually, we define an~\textit{upset}.

For $Q\subseteq P$ we denote
\begin{align*}
    \upset_\leq Q &\coloneqq \{p\in P\mid \exists\,q\in Q:q\leq p\}\,,\quad\text{(upset generated by $Q$)}\\
    \downset_\leq Q &\coloneqq \{p\in P\mid \exists\,q\in Q:p\leq q\}\,.\quad\text{(downset generated by $Q$)}
\end{align*}

Let $(Q,\,\preceq)$ be a poset. We say that $(Q,\,\preceq)$ \textit{embeds} into $(P,\,\leq)$ if there is an injective map $\nu$ from $Q$ to $P$, which is order-preserving, that is
\begin{equation*}
    p\preceq q \implies \nu(p)\leq\nu(q)
\end{equation*}
for all $p,\,q\in Q$.

Let $(P_1,\,\leq_{P_1})$, $(P_2,\,\leq_{P_2})$ be posets. We define the \textit{product ordering} $\preceq$ on $P_1\times P_2$ by the following rule:
\begin{equation*}
    \left[(p_1,\,p_2)\preceq(p_1',\,p_2')\right] \iff \left[(p_1\leq_{P_1} p_1') \text{ and } (p_2\leq_{P_2} p_2')\right]\,.
\end{equation*}
To denote product posets, we will also use the notation
\begin{equation*}
    \prod_{i\,=\,1}^n (P_i,\,\leq_i) \coloneqq (P_1,\,\leq_1)\times\dots\times(P_n,\,\leq_n)\,.
\end{equation*}
\end{definition}
\begin{observation}\label{thm:closureoperators}
Let $(P,\,\leq)$ be a poset. The mappings $\upset_\leq$ and $\downset_\leq$ are algebraic closure operators on $P$. The closed sets are the upsets and the downsets respectively.

The mappings $\Clo^k$, $\sClo^k$, and $\qpp[k]{-}$ are algebraic closure operators on $\OO_{A}^k$, $\OO_{A}^k$, and $R_A^k$ respectively.
\end{observation}
\begin{observation}
Let $A$ be a nonempty set and $\cls$ an algebraic closure operator on $A$. For $B\subseteq A$, the mapping $\cls':2^B\to 2^B$ defined by the following rule:
\begin{equation*}
    \cls'C = (\cls{C})\cap B\,,
\end{equation*}
is an algebraic closure operator on $B$.
\end{observation}
\begin{definition}
Let $k\in\N$. By $\CR^k_{1-6}\subseteq\CR^k$ we denote the (finite) set of all $k$-sorted canonical relations of the types (c1)--(c6) and by $\CR^k_{7}\subseteq\CR^k$ we denote the set of all $k$-sorted canonical relations of the type (c7).
\end{definition}
\begin{observation}\label{thm:predtotuple}
Let $k\in\N$. We define a mapping $\mu: \CR_7^k\to \N_0^{2k}$ by the following rule.
\begin{multline*}
    \mu\left(\disj{\begin{array}{r@{\mskip\medmuskip}c@{\mskip\medmuskip}c@{\mskip\medmuskip}c@{\mskip\medmuskip}l}
        x_1^{1} = 0 \,\lor & \cdots & \lor\, x_{m_1}^{1} = 0 \lor y_1^{1} = 1 \,\lor & \cdots & \lor\, y_{n_1}^{1} = 1  \\
        x_1^{2} = 0 \,\lor & \cdots & \lor\, x_{m_2}^{2} = 0 \lor y_1^{2} = 1 \,\lor & \cdots & \lor\, y_{n_2}^{2} = 1\\
        &\vdots&&\vdots& \\
        x_1^{k} = 0 \,\lor & \cdots & \lor\, x_{m_k}^{k} = 0 \lor y_1^{k} = 1 \,\lor & \cdots & \lor\, y_{n_k}^{k} = 1
    \end{array}}\right) \coloneqq \\\coloneqq(m_1,\,n_1,\,m_2,\,n_2,\dots,\,m_k,\,n_k)\,,
\end{multline*}
where $m_i,\,n_i\in\N_0$ and either $m_i = 0$ or $n_i = 0$ for all $i\in E_k$. Then
\begin{equation*}
    \mu(\rho)\preceq \mu(\rho') \iff \qpp{\rho}\subseteq\qpp{\rho'}
\end{equation*}
for all $\rho,\,\rho'\in\CR_7^k$.
\end{observation}
\begin{theorem}\label{thm:embedding}
Let $k\in\N$. The poset
\begin{equation*}
    (\mathfrak{L}_{\qpp[k]{-}}(R_{0,\,1}^k),\,\subseteq)
\end{equation*}
embeds into the poset
\begin{equation*}
    (2^{\CR_{1-6}^k},\,\subseteq)\times (\mathfrak{L}_{\downset_\preceq}(\N_0^{2k}),\,\subseteq)\,,
\end{equation*}
where $\preceq$ is the product ordering on $\N_0^{2k}$.
\end{theorem}
\begin{proof}
Since a composition of embeddings is an embedding itself, we will prove the desired claim in steps.

By Theorem \ref{thm:CR}, the poset $(\mathfrak{L}_{\qpp[k]{-}}(R_{\{0,\,1\}}^k),\,\subseteq)$ is isomorphic to the poset $(\mathfrak{L}_{\qpp[k]{-}\,\cap\,\CR^k}(\CR^k),\,\subseteq)$.

Since for every closed subset $S$ of $\CR^k$ we may consider the partition
\begin{equation*}
    S = (S\cap \CR_{1-6}^k)\cup (S\cap \CR_{7}^k)\,,
\end{equation*}
we have the embedding of $(\mathfrak{L}_{\qpp[k]{-}\,\cap\,\CR^k}(\CR^k),\,\subseteq)$ into the poset
\begin{equation*}
    (2^{\CR_{1-6}^k},\,\subseteq)\times (\mathfrak{L}_{\qpp[k]{-}\,\cap\,\CR^k_7}(\CR^k_7),\,\subseteq)\,.
\end{equation*}

Thus, it remains to be shown that the poset $(\mathfrak{L}_{\qpp[k]{-}\,\cap\,\CR^k_7}(\CR^k_7),\,\subseteq)$ embeds into the poset $(\mathfrak{L}_{\downset_\preceq}(\N_0^{2k}),\,\subseteq)$.

Considering the mapping $\mu: \CR_7^k\to \N_0^{2k}$ from Observation~\ref{thm:predtotuple}, we define a~mapping $\nu:\mathfrak{L}_{\qpp[k]{-}\,\cap\,\CR^k_7}(\CR^k_7)\to\mathfrak{L}_{\downset_\preceq}(\N_0^{2k})$ by the following rule:
\begin{equation*}
    \nu(S) = \{\mu(\rho)\mid \rho\in S\}\,.
\end{equation*}
It follows from Observation~\ref{thm:predtotuple} that $\nu$ is well-defined and, in fact, the desired embedding.
\end{proof}
\subsection{Inductive structure of the lattice}\label{sec:inductivestructure}
Let us recall that by the Galois connection (Theorem \ref{thm:Galois}), the posets
\begin{equation*}
    (\mathfrak{L}_{\qpp[k]{-}}(R_{\{0,\,1\}}^k),\,\subseteq)\quad\text{and}\quad(\mathfrak{L}_{\sClo^k}(\OO_{\{0,\,1\}}^k),\,\subseteq)
\end{equation*}
are antiisomorphic, that is, there is a bijective, order-reversing mapping between them.

We will now demonstrate that the embedding described in Theorem \ref{thm:embedding} has implications not only for the lattice of quantified relational clones but also for the bigger lattice of (relational) clones. Moreover, as the following result suggests, this lattice has quite a simple inductive structure.
\begin{theorem}\label{thm:indstr}
Let $k\in\N$, $k\geq 2$. The poset 
\begin{equation*}
    (\mathfrak{L}_{\Clo^k}(\OO_{\{0,\,1\}}^k),\,\subseteq)
\end{equation*} embeds into the product poset
\begin{equation*}
    (\mathfrak{L}_{\sClo^k}(\OO_{\{0,\,1\}}^k),\,\subseteq)\times\prod_{i\,=\,1}^{2k}\left((\mathfrak{L}_{\Clo^{k-1}}(\OO_{\{0,\,1\}}^{k-1})\cup\{\emptyset\},\,\subseteq)\times (\{0,\,1\},\,\leq)\right)\,,
\end{equation*}
where $\leq$ denotes the standard linear ordering on $\{0,\,1\}$.
\end{theorem}
\begin{proof}
Let $\mathcal{C}$ be a $k$-clone on $\{0,\,1\}$. For all $b\in\{0,\,1\}$ and $i\in E_k$ we denote
\begin{equation*}
    \mathcal{C}_{i,b} \coloneqq \{\boldsymbol{f}\in\mathcal{C}\mid \boldsymbol{f}^{(i)}\text{ is constant }b\}\,.
\end{equation*}
Notice that we may express $\mathcal{C}$ as the following union of $2k+1$ sets:
\begin{equation}
    \mathcal{C} = \sClo^k\mathcal{C} \cup \bigcup_{i\,\in\, E_k} (\mathcal{C}_{i,0} \cup \mathcal{C}_{i,1})\,. \tag{$\ast$}
\end{equation}
This decomposition highlights the main idea behind the embedding: since the $k$-operations in $\mathcal{C}_{i,b}$ effectively reduce to $(k-1)$-operations, it is natural to expect the sets $\mathcal{C}_{i,b}$ to correspond to $(k-1)$-clones. To formalize this intuition, we introduce the following notation and prove Claim~\ref{thm:claiminduction}, which characterizes the structure of the sets $\mathcal{C}_{i,b}$.

Let $\mathcal{P}^{k-1}$ be the $(k-1)$-clone generated by the empty set (i.e., the clone of $(k-1)$-projections), and let
\begin{equation*}
    \mathcal{C}_{i,b}' \coloneqq \{(\boldsymbol{f}^{(1)},\dots,\,\boldsymbol{f}^{(i-1)},\,\boldsymbol{f}^{(i+1)},\dots,\,\boldsymbol{f}^{(k)})\mid \boldsymbol{f}\in \mathcal{C}_{i,b}\}\,.
\end{equation*}
\begin{claim}\label{thm:claiminduction}
For each $b\in\{0,\,1\}$ and $i\in E_k$, the set $\mathcal{C}_{i,b}'$ is either empty, a $(k-1)$-clone, or a $(k-1)$-clone without all $(k-1)$-projections.
\end{claim}
\begin{proof}[Proof of Claim \ref{thm:claiminduction}.]
Assume that $\mathcal{C}_{i,b}'$ is nonempty. The following observations are sufficient to prove the claim:
\begin{enumerate}
    \item The set $\mathcal{C}_{i,b}'$ is closed under compositions.
    \item The set $\mathcal{C}_{i,b}'\cup\mathcal{P}^{k-1}$ is a clone.
    \item If $\mathcal{C}_{i,b}'$ contains any $(k-1)$-projection, then it contains all $(k-1)$-projections.
\end{enumerate}
As all the points are straightforward to verify, we provide just the following brief comments to clarify the key reasoning:
\begin{itemize}
    \item (1) follows from the fact that $\mathcal{C}_{i,b}$ is itself closed under compositions.
    \item (2) and (3) are direct consequences of (1) combined with the observation that the composition of $k$-projections with the $k$-operations in $\mathcal{C}_{i,b}$ results in either a~$k$-projection or a~$k$-operation within $\mathcal{C}_{i,b}$.
    
    Specifically, if $\boldsymbol{f}'\in \mathcal{C}_{i,b}'$ is a $(k-1)$-projection, then every $(k-1)$-projection can be obtained by composing the corresponding $k$-operation $\boldsymbol{f} \in \mathcal{C}_{i,b}$ with suitable $k$-projections. \qedhere
\end{itemize}
\end{proof}
Based on this claim, let us define the following mapping for each $b\in\{0,\,1\}$ and $i\in E_k$:
\begin{equation*}
    \nu_{i,b}: 
    (\mathfrak{L}_{\Clo^k}(\OO_{\{0,\,1\}}^k),\,\subseteq)\to
    \left((\mathfrak{L}_{\Clo^{k-1}}(\OO_{\{0,\,1\}}^{k-1})\cup\{\emptyset\},\,\subseteq)\times (\{0,\,1\},\,\leq)\right)
\end{equation*}
given by
\begin{equation*}
    \nu_{i,b}(\mathcal{C}) = \left\{\begin{array}{ll}
        (\emptyset, 0)\,, & \text{ if }\mathcal{C}_{i,b}'\text{ is empty},\\
        (\mathcal{C}_{i,b}'\cup\mathcal{P}^{k-1},\, 0)\,, & \text{ if }\mathcal{C}_{i,b}' \text{ does not contain $(k-1)$-projections},\\
        (\mathcal{C}_{i,b}',\, 1)\,, & \text{ if }\mathcal{C}_{i,b}'\text{ contains $(k-1)$-projections}.
    \end{array}\right.
\end{equation*}
Notice that this mapping is well-defined, as the cases in the definition of $\nu_{i,b}$ account for all possibilities described in Claim \ref{thm:claiminduction}. Furthermore, it is order-preserving because the definition of $\nu_{i,b}$ ensures that $\mathcal{C} \subseteq \mathcal{D}$ implies $\mathcal{C}_{i,b}' \subseteq \mathcal{D}_{i,b}'$, which in turn guarantees $\nu_{i,b}(\mathcal{C}) \preceq \nu_{i,b}(\mathcal{D})$ for all $k$-clones $\mathcal{C}$ and~$\mathcal{D}$.

We now define the desired embedding:
\begin{equation*}
    \nu: \mathcal{C} \mapsto \left(\sClo^k\mathcal{C},\,\nu_{1,0}(\mathcal{C}),\,\nu_{1,1}(\mathcal{C}),\dots,\,\nu_{k,0}(\mathcal{C}),\,\nu_{k,1}(\mathcal{C})\right)\,.
\end{equation*}
It is clear that $\nu$ is order-preserving. Finally, we observe that $\nu$ is injective. Indeed, by the decomposition $(\ast)$ and the definition of $\nu_{i,b}$, every $k$-clone $\mathcal{C}$ is uniquely determined by its surjective part $\sClo^k\mathcal{C}$ and the sets $\mathcal{C}_{i,b}$. Since $\nu_{i,b}$ encodes the sets $\mathcal{C}_{i,b}$ uniquely, the image $\nu(\mathcal{C})$ completely determines $\mathcal{C}$, completing the proof.
\end{proof}
As a direct consequence of the established results, we obtain the following fundamental corollary.
\begin{corollary}\label{thm:indstrcor}
Let $k\in\N$, $k\geq 2$. The poset 
\begin{equation*}
    (\mathfrak{L}_{\Clo^k}(\OO_{\{0,\,1\}}^k),\,\subseteq)
\end{equation*} embeds into the product poset
\begin{multline*}
    (2^{\CR_{1-6}^k},\,\supseteq)\times (\mathfrak{L}_{\downset_\preceq}(\N_0^{2k}),\,\supseteq)\,\times\\ \times \prod_{i\,=\,1}^{2k}\left((\mathfrak{L}_{\Clo^{k-1}}(\OO_{\{0,\,1\}}^{k-1})\cup\{\emptyset\},\,\subseteq)\times (\{0,\,1\},\,\leq)\right)\,,
\end{multline*}
where $\preceq$ denotes the product ordering on $\N_0^{2k}$, and $\leq$ is the standard linear ordering on $\{0,\,1\}$.
\end{corollary}
\begin{proof}
We combine the results of Theorem \ref{thm:indstr}, Theorem \ref{thm:Galois}, and Theorem~\ref{thm:embedding}.
\end{proof}
Although Theorem~\ref{thm:indstr} and Corollary~\ref{thm:indstrcor} may initially appear symbolically dense, they provide valuable insight into the structure of the lattice of (surjective) clones. For instance, its countability follows as a straightforward consequence.

Another key implication of this embedding is the fact that every $k$-clone is finitely generated. To observe this, we define a few more order theory concepts that allow us to relate a particular property of the lattice of clones to a~structural property of the clones themselves.
\begin{definition}[Order theory concepts II]
Let $(P,\,\leq)$ be a~poset. We say $P$ satisfies the~\textit{ascending chain condition} (ACC) if for every ascending chain, that is, an infinite sequence $p_1,\,p_2,\,p_3,\dots\in P$ satisfying
\begin{equation*}
    p_1\leq p_2\leq p_3\leq\dots\,,
\end{equation*}
there is $m\in\N$ such that $p_{m} = p_{n}$ for all $n\geq m$.
\end{definition}
\begin{observation}\label{thm:posetemb}
Let $(Q,\,\preceq)$ and $(P,\,\leq)$ be posets.
\begin{enumerate}
    \item If $(Q,\,\preceq)$ embeds into $(P,\,\leq)$, and $(P,\,\leq)$ satisfies ACC, then $(Q,\,\preceq)$ satisfies ACC.
    \item If $(Q,\,\preceq)$ and $(P,\,\leq)$ satisfy ACC, then $(Q,\,\preceq)\times(P,\,\leq)$ satisfies ACC.
    \item If the set $P$ is finite, then $(P,\,\leq)$ satisfies ACC.
\end{enumerate}
\end{observation}
The following lemma is a canonical result in algebra and a~key component of the presented proof of Taimanov's theorem. It establishes a connection between ACC -- a~structural property of the poset of closed sets -- with the fact that every set is finitely generated. For further details, we refer the reader to~\cite[Chapter 6]{Atiyah1969}, \cite{Hodges1974}, or any standard text on Noetherian modules, rings, or algebras.

While this result is often stated under the Axiom of Choice (or some of its weaker variants), the following formulation avoids the need for AC by assuming countability. If countability is not assumed, AC becomes necessary; see~\cite{Hodges1974}.
\begin{lemma}\label{thm:ACCFG}
Let $A$ be a~nonempty countable set and let $\cls$ be an algebraic closure operator on $A$. The poset $(\mathfrak{L}_{\cls}(A),\,\subseteq)$ satisfies ACC if and only if every element of $\mathfrak{L}_{\cls}(A)$ is finitely generated with respect to $\cls$.
\end{lemma}
Intending to use Observation \ref{thm:posetemb} along with Theorem \ref{thm:embedding}, we present one final lemma before directly proceeding to prove Taimanov's theorem.
\begin{lemma}\label{thm:downsetsofNm}
Let $m\in\N$ and let $\preceq$ be the product ordering on $\N_0^m$. The~poset $(\mathfrak{L}_{\downset_\preceq}(\N_0^{m}),\,\supseteq)$ satisfies ACC.
\end{lemma}
\begin{theorem}[Taimanov] \label{thm:taimanov}
Let $k\in\N$. Every $k$-clone on $\{0,\,1\}$ is finitely generated with respect to $\Clo^k$.
\end{theorem}
\begin{proof}
By Lemma~\ref{thm:ACCFG}, it suffices to show that the poset $(\mathfrak{L}_{\Clo^k}(\OO_{\{0,\,1\}}^k),\,\subseteq)$ satisfies ACC. We prove this by induction on $k$. 

For $k=1$, we have ACC by Post's lattice theorem. Alternatively, this follows from the fact that $(\mathfrak{L}_{\Clo^1}(\OO_{\{0,\,1\}}^1),\,\subseteq)$ embeds into the poset
\begin{equation*}
    (\mathfrak{L}_{\sClo^1}(\OO_{\{0,\,1\}}^1),\,\subseteq)\times (2^{\{\mathfrak{0},\,\mathfrak{1}\}},\,\subseteq)\,,
\end{equation*}
which, in turn, embeds into
\begin{equation*}
    (2^{\CR_{1-6}^1},\,\supseteq)\times (\mathfrak{L}_{\downset_\preceq}(\N_0^{2}),\,\supseteq)\times (2^{\{\mathfrak{0},\,\mathfrak{1}\}},\,\subseteq)\,.
\end{equation*}
The first embedding is a consequence of the fact that there are only two non-surjective $1$-operations $\mathfrak{0}$ and $\mathfrak{1}$ (up to dummy variables), so any $1$-clone is determined by a surjective $1$-clone and a subset of $\{\mathfrak{0},\,\mathfrak{1}\}$. The second embedding is then given by Theorem~\ref{thm:embedding} and the Galois connection~\ref{thm:Galois}.

By Observation~\ref{thm:posetemb}, the poset $(\mathfrak{L}_{\Clo^1}(\OO_{\{0,\,1\}}^1),\,\subseteq)$ satisfies ACC since it embeds into a product of two finite posets and a poset $(\mathfrak{L}_{\downset_\preceq}(\N_0^2),\,\supseteq)$, which satisfies ACC by Lemma~\ref{thm:downsetsofNm}.

Now assume that $k \geq 2$, and suppose that the poset $(\mathfrak{L}_{\Clo^{k-1}}(\OO_{\{0,\,1\}}^{k-1}),\,\subseteq\nolinebreak)$ satisfies ACC (induction hypothesis).

By Corollary~\ref{thm:indstrcor}, the poset $(\mathfrak{L}_{\Clo^k}(\OO_{\{0,\,1\}}^k),\,\subseteq)$ embeds into
\begin{multline*}
    (2^{\CR_{1-6}^k},\,\supseteq)\times (\mathfrak{L}_{\downset_\preceq}(\N_0^{2k}),\,\supseteq)\,\times\\ \times \prod_{i\,=\,1}^{2k}\left((\mathfrak{L}_{\Clo^{k-1}}(\OO_{\{0,\,1\}}^{k-1})\cup\{\emptyset\},\,\subseteq)\times (\{0,\,1\},\,\leq)\right)\,.
\end{multline*}
To complete the proof, we note the following:
\begin{itemize}
    \item $(2^{\CR_{1-6}^k},\,\supseteq)$ is finite,
    \item $(\mathfrak{L}_{\downset_\preceq}(\N_0^{2k}),\,\supseteq)$ satisfies ACC by Lemma~\ref{thm:downsetsofNm},
    \item $(\mathfrak{L}_{\Clo^{k-1}}(\OO_{\{0,\,1\}}^{k-1})\cup\{\emptyset\},\,\subseteq)$ satisfies ACC by the induction hypothesis (the inclusion of $\emptyset$ does not affect this property),
    \item $(\{0,\,1\},\,\leq)$ is finite.
\end{itemize}
By Observation~\ref{thm:posetemb}, the product of these posets satisfies ACC, and therefore, $(\mathfrak{L}_{\Clo^k}(\OO_{\{0,\,1\}}^k),\,\subseteq)$ satisfies ACC as well.
\end{proof}
\begin{corollary}
Let $k\in\N$. The set $\mathfrak{L}_{\Clo^k}(\OO_{\{0,\,1\}}^k)$ is countable.
\end{corollary}

\section{Galois connection}\label{sec:Galois}
This section presents the proof of Theorem \ref{thm:Galois}. Although the main structure is pretty canonical and based on the known results \cite{Geiger1968,Bodnarchuk1969I,Bodnarchuk1969II,Romov1973,Borner2009}, our aim is to provide a comprehensive proof with a focus on extending the $\sPol$-$\Inv$ result from \cite{Borner2009} to the multi-sorted case. We begin by establishing some conventions that will be used throughout. 
\begin{convention}
In this section, we fix $k\in\N$, $l\in\N$, and a~finite nonempty set $A = \{1,\dots,\,l\}$.
\end{convention}
The first part of the theorem, that is, the fact that the mappings $\Inv$ and $\sPol$ form a Galois connection follows directly from the definition.
\begin{observation}\label{thm:galoisproperty}
For all $\mathcal{F}'\subseteq\mathcal{F}\subseteq \OO_A^k$ and all $S'\subseteq S \subseteq R_A^k$, we have
\begin{align*}
    \Inv \mathcal{F} &\subseteq \Inv\mathcal{F}'\,, & \mathcal{F}&\subseteq\sPol\Inv\mathcal{F}\,, & \Inv\sPol\Inv\mathcal{F} &= \Inv\mathcal{F}\,, \\
    \sPol S &\subseteq \sPol S'\,, & S&\subseteq\Inv\sPol S\,, & \sPol\Inv\sPol S &= \sPol S\,,
\end{align*}
\end{observation}
To understand the structure of the closed sets in this Galois connection, we rely on the observation that any $n$-ary $k$-operation on $A = \{1,\dots,\,l\}$ can be represented as a $(l^n \cdot k)$-tuple by listing its values on all $n$-tuples of inputs in a fixed order. This is formalized in the following remark.
\begin{remark}\label{def:operationstuples}
Let $n\in\N$ and let $f:A^n\to A$ be an $n$-ary function. By $\widetilde{f}$, we denote the tuple
\begin{equation*}
    \widetilde{f} \coloneqq \left.\begin{pmatrix}
        f(1,\,1,\dots,\,1,\,1)\\
        f(1,\,1,\dots,\,1,\,2)\\
        \vdots\\
        f(l-1,\,l,\dots,\,l,\,l)\\
        f(l,\,l,\dots,\,l,\,l)
    \end{pmatrix}\!\right\}{\scriptstyle l^n}\,,
\end{equation*}
where arguments of $f$ are all $n$-tuples of elements of $A^n$ in a standard lexicographical order.

Given an~$n$-ary $k$-operation $\boldsymbol{f}\in\OO_{A,n}^k$, by $\widetilde{\boldsymbol{f}}$, we denote the following $(l^n\cdot k)$-tuple:
\begin{equation*}
    \widetilde{\boldsymbol{f}} \coloneqq \left.\begin{pmatrix}
        \widetilde{\boldsymbol{f}^{(1)}}\\
        \widetilde{\boldsymbol{f}^{(2)}}\\
        \vdots\\
        \widetilde{\boldsymbol{f}^{(k)}}
    \end{pmatrix}\!\right\}{\scriptstyle l^n\cdot k}\,.
\end{equation*}

Finally, for a set of $k$-operations $\mathcal{F}\subseteq\OO_A^k$, by $\widetilde{\mathcal{F}}_n$, we denote the following $(l^n \cdot k)$-ary relation:
\begin{equation*}
    \widetilde{\mathcal{F}}_n \coloneqq \{\widetilde{\boldsymbol{f}}\mid \boldsymbol{f}\in\mathcal{F}\cap \OO_{A,n}^k\}
\end{equation*}
in which the first $l^n$ variables are considered to be of the first sort, the second $l^n$ variables are considered to be of the second sort, and this pattern is followed for each successive group of $l^n$ variables.
\end{remark}
\begin{observation}\label{thm:opinclusion}
Let $\boldsymbol{f}\in\OO_A^k$ be an $n$-ary operation and let $\mathcal{F}\subseteq\OO_A^k$. Then $\boldsymbol{f}\in\mathcal{F}$ if and only if $\widetilde{\boldsymbol{f}}\in\widetilde{\mathcal{F}}_n$.
\end{observation}
\subsection{Closed sets of operations}
\begin{theorem}\label{thm:polinv}
Let $\mathcal{F}\subseteq\OO_{A}^k$. Then
\begin{align*}
    \Pol\Inv\mathcal{F} &= \Clo \mathcal{F}\,,\\
    \sPol\Inv\mathcal{F} &= \sClo \mathcal{F}\,.
\end{align*}
\end{theorem}
\begin{proof}
We will only prove that $\Pol\Inv\mathcal{F} = \Clo \mathcal{F}$ as $\sPol\Inv\mathcal{F} = \sClo \mathcal{F}$ is an~immediate corollary obtained by considering only the surjective $k$-operations.

First, we may notice that $\Clo\Pol S = \Pol S$ for any set $S\subseteq R_A^k$. Indeed, all $k$-projections preserve all relations, and if some $k$-operations preserve a~certain relation, their composition preserves such relation as well. Due to this, we have the inclusion
\begin{equation*}
    \Pol\Inv\mathcal{F} \supseteq \Clo \mathcal{F}
\end{equation*}
as $\Clo \mathcal{F}$ is the smallest $k$-clone containing $\mathcal{F}$ and $\Pol\Inv\mathcal{F}$ is some $k$-clone containing $\mathcal{F}$.

For the converse inclusion, we consider any $\boldsymbol{f}\in \Pol\Inv\mathcal{F}$. Let $n$ be the arity of $\boldsymbol{f}$. Let us prove the following claim.

\begin{claim}\label{thm:galoiscloclaim}
$\widetilde{\Clo\mathcal{F}}_n \in \Inv\mathcal{F}$.
\end{claim}
\begin{proof}[Proof of Claim \ref{thm:galoiscloclaim}.]
Any $\boldsymbol{a}_1,\dots,\,\boldsymbol{a}_m\in \widetilde{\Clo\mathcal{F}}_n$ correspond to some $n$-ary $k$-operations $\boldsymbol{h}_1,\dots,\,\boldsymbol{h}_m\in \Clo\mathcal{F}$, thus, for any $m$-ary $k$-operation $\boldsymbol{g}\in\mathcal{F}$, the tuple $\boldsymbol{g}(\boldsymbol{a}_1,\dots,\,\boldsymbol{a}_m)$ corresponds to the composition of $\boldsymbol{g}$ with $\boldsymbol{h}_1,\dots,\,\boldsymbol{h}_m$, which is already contained in $\Clo\mathcal{F}$ as it is a clone. Thus $\boldsymbol{g}(\boldsymbol{a}_1,\dots,\,\boldsymbol{a}_m)\in\widetilde{\Clo\mathcal{F}}_n$.
\end{proof}
Now, applying $\Pol$ on both sides of $\widetilde{\Clo\mathcal{F}}_n \in \Inv\mathcal{F}$, we get
\begin{equation*}
    \boldsymbol{f}\in\Pol\Inv\mathcal{F} \subseteq \Pol \{\widetilde{\Clo\mathcal{F}}_n\}\,.
\end{equation*}
In particular, this means that $\widetilde{\boldsymbol{f}}\in\widetilde{\Clo\mathcal{F}}_n$ as $\widetilde{\boldsymbol{f}}$ is just the $k$-operation $\boldsymbol{f}$ evaluated at the tuples in $\widetilde{\Clo\mathcal{F}}_n$ corresponding to the projections. Therefore, $\boldsymbol{f}\in\Clo\mathcal{F}$ by Observation \ref{thm:opinclusion}. This finishes the proof of the inclusion
\begin{equation*}
    \Pol\Inv\mathcal{F} \subseteq \Clo \mathcal{F}\,.
\end{equation*}
\end{proof}
\subsection{Closed sets of relations/predicates}
\begin{lemma}\label{thm:invisrelclo}
Let $\mathcal{F}\subseteq\OO_A^k$. Then $\Inv{\mathcal{F}}$ is a~relational clone, that is
\begin{equation*}
    \pp{\Inv\mathcal{F}} = \Inv\mathcal{F}\,.
\end{equation*}
Moreover, if $\mathcal{F}$ contains only surjective operations, then $\Inv{\mathcal{F}}$ is a~quantified relational clone, that is
\begin{equation*}
    \qpp{\Inv\mathcal{F}} =\Inv\mathcal{F}\,.
\end{equation*}
\end{lemma}
\begin{proof}
Since empty relations and equalities are preserved by every surjective $k$-operation we only need to check that $\Inv\mathcal{F}$ is closed under the elementary operations (Definition \ref{def:eo} and Lemma \ref{thm:eo}). The arguments are mostly simple, we outline the basic ideas. Let $\rho,\,\rho'\in \Inv\mathcal{F}$.
\begin{itemize}
    \item Closedness under (eo1) and (eo2) is clear.
    \item If $\sigma$ is defined from $\rho$ by identification of variables (eo3), then $\sigma\in\Inv\mathcal{F}$ since operations are applied coordinate-wise.
    \item If $\sigma$ is defined from $\rho$ and $\rho'$ by conjunction (eo6), we have $\sigma=\rho\cap\rho'$. Since both $\rho$ and $\rho'$ are invariant under all $k$-operations in $\mathcal{F}$, applying any such $k$-operation to tuples from the intersection $\rho \cap \rho'$ will result in a tuple that also belongs to the intersection. Thus, $\sigma = \rho\cap\rho' \in\Inv\mathcal{F}$.
    \item If $\sigma$ is defined from $\rho$ by existential quantification of a variable, then $\sigma\in\Inv\mathcal{F}$ since operations are applied coordinate-wise. It follows that $\Inv\mathcal{F}$ is closed under composition (eo4).
\end{itemize}
This finishes the proof of $\pp{\Inv\mathcal{F}} = \Inv\mathcal{F}$.

Now assume that all operations in $\mathcal{F}$ are surjective. The main reason why $\Inv\mathcal{F}$ is closed under universal quantification (eo5) is as follows:
\begin{itemize}
    \item Let $\boldsymbol{f}\in\mathcal{F}$ be an $n$-ary operation, $\rho\in\Inv\mathcal{F}$, and
    \begin{equation*}
        \sigma(x_1,\dots,\,x_m) = \forall y\, \rho(y,\,x_1,\,x_2,\dots,\,x_m)\,.
    \end{equation*}
    If $\boldsymbol{a}_1,\dots,\,\boldsymbol{a}_m\in\sigma$, then 
    \begin{equation*}
        \begin{pmatrix}
            b_1 \\
            \boldsymbol{a}_1^{(1)}\\
            \boldsymbol{a}_1^{(2)}\\
            \vdots\\
            \boldsymbol{a}_1^{(m)}
        \end{pmatrix},\begin{pmatrix}
            b_2 \\
            \boldsymbol{a}_2^{(1)}\\
            \boldsymbol{a}_2^{(2)}\\
            \vdots\\
            \boldsymbol{a}_2^{(m)}
        \end{pmatrix},\dots,\,\begin{pmatrix}
            b_n \\
            \boldsymbol{a}_n^{(1)}\\
            \boldsymbol{a}_n^{(2)}\\
            \vdots\\
            \boldsymbol{a}_n^{(m)}
        \end{pmatrix}\in\rho
    \end{equation*}
    for all $b_1,\,b_2,\dots,\,b_n\in A$. Since $\boldsymbol{f}$ is surjective operations which preserves $\rho$, this also means that
    \begin{equation*}
        \begin{pmatrix}
            b \\ f(\boldsymbol{a}_1^{(1)},\,\boldsymbol{a}_2^{(1)},\dots,\,\boldsymbol{a}_n^{(1)}) \\
            f(\boldsymbol{a}_1^{(2)},\,\boldsymbol{a}_2^{(2)},\dots,\,\boldsymbol{a}_n^{(2)}) \\
            \vdots \\
            f(\boldsymbol{a}_1^{(m)},\,\boldsymbol{a}_2^{(m)},\dots,\,\boldsymbol{a}_n^{(m)})
        \end{pmatrix}\in\rho
    \end{equation*}
    for all $b\in A$. Therefore,
    \begin{equation*}
        \begin{pmatrix}f(\boldsymbol{a}_1^{(1)},\,\boldsymbol{a}_2^{(1)},\dots,\,\boldsymbol{a}_n^{(1)}) \\
            f(\boldsymbol{a}_1^{(2)},\,\boldsymbol{a}_2^{(2)},\dots,\,\boldsymbol{a}_n^{(2)}) \\
            \vdots \\
            f(\boldsymbol{a}_1^{(m)},\,\boldsymbol{a}_2^{(m)},\dots,\,\boldsymbol{a}_n^{(m)})
        \end{pmatrix}\in\sigma\,.
    \end{equation*}
    Hence, $\Inv\mathcal{F}$ is closed under universal quantification of a variable (eo5).\qedhere
\end{itemize}
\end{proof}
\begin{lemma}\label{thm:invpolstep1}
Let $S\subseteq R_A^k$ and let $n\in\N$. Then $\widetilde{\Pol S}_n\in\pp{S}$. 
\end{lemma}
\begin{proof}
We will verify that $\widetilde{\Pol S}_n = \sigma$, where
\begin{multline*}
    \sigma(x_{0,\dots,\,0}^1,\dots,\,x_{l-1,\dots,\,l-1}^1,\dots,\,x_{0,\dots,\,0}^k,\dots,\,x_{l-1,\dots,\,l-1}^k) =\\= \bigwedge_{\substack{\rho\,\in\, S\\ \boldsymbol{a}_1,\dots,\,\boldsymbol{a}_n\,\in\,\rho}} \rho\left(x_{\boldsymbol{a}_1^{(1)},\,\boldsymbol{a}_2^{(1)},\dots,\,\boldsymbol{a}_n^{(1)}}^{i_{\rho,\,1}},\dots,\,x_{\boldsymbol{a}_1^{(m_\rho)},\,\boldsymbol{a}_2^{(m_\rho)},\dots,\,\boldsymbol{a}_n^{(m_\rho)}}^{i_{\rho,\,m_\rho}}\right)\,.
\end{multline*}
Here, for $\rho \in S$, the arity of $\rho$ is denoted by $m_\rho$, and the sort of its $j$-th variable by $i_{\rho,\,j}$. Let us remark that although this conjunction is infinite when $S$ is infinite, the domain is finite, so we can restrict to a finite number of terms without changing result.

The reason behind the equality $\widetilde{\Pol S}_n = \sigma$ is as follows. We may consider the tuples in $\sigma$ to uniquely correspond to images of $n$-ary $k$-operations in some $\mathcal{F}\subseteq\OO_A^k$ according to Remark~\ref{def:operationstuples}; that is, $\sigma = \widetilde{\mathcal{F}}_n$. Each term of the form
\begin{equation*}
    \rho\left(x_{\boldsymbol{a}_1^{(1)},\,\boldsymbol{a}_2^{(1)},\dots,\,\boldsymbol{a}_n^{(1)}}^{i_{\rho,\,1}},\dots,\,x_{\boldsymbol{a}_1^{(m_\rho)},\,\boldsymbol{a}_2^{(m_\rho)},\dots,\,\boldsymbol{a}_n^{(m_\rho)}}^{i_{\rho,\,m_\rho}}\right)
\end{equation*}
then ensures that, for every $\boldsymbol{f}\in\mathcal{F}$, we have
\begin{equation*}
    \boldsymbol{f}\left(\boldsymbol{a}_1,\dots,\,\boldsymbol{a}_n\right)\in\rho\,.
\end{equation*}
This means that $\mathcal{F}$ contains exactly all the $n$-ary $k$-operations that preserve all the relations in $S$, and thus $\sigma = \widetilde{\mathcal{F}}_n = \widetilde{\Pol S}_n$.
\end{proof}
\begin{lemma}\label{thm:invpolstep2}
Let $S\subseteq R_A^k$ and let $\rho\in \Inv\Pol S$ be a relation of size $n$. Then $\rho\in\pp{\widetilde{\Pol S}_n}$.
\end{lemma}
\begin{proof}
For the purpose of this proof, we assume that the tuples of every relation are ordered so that a $p$-ary relation of size $q$ can be viewed as a $p \times q$ matrix. To that end, we fix an arbitrary order of tuples of $\rho$. We also let $m$ be the arity of $\rho$.

Using the notation from Definition~\ref{def:operations}, we let $\boldsymbol{p}_{n,i}^k$ denote the $n$-ary projection $k$-operation to the $i$-th coordinate. We now define the following $(k\cdot l^n)$-ary relation of size $n$:
\begin{equation*}
    \theta = \begin{pmatrix}
        \widetilde{\boldsymbol{p}_{n,1}^k} & \widetilde{\boldsymbol{p}_{n,2}^k} & \cdots & \widetilde{\boldsymbol{p}_{n,n}^k}
    \end{pmatrix}\,.
\end{equation*}
Observe that the rows of $\theta$ are structured as $k$ blocks, each containing all $n$-tuples of elements from $A$, arranged in lexicographical order. By this observation, for each $i\in E_m$, there is a unique $r_i\in E_{kl^n}$ such that
\begin{itemize}
    \item when viewed as a matrix, the $i$-th row of $\rho$ is the same as the $r_i$-th row of $\theta$,
    \item the $i$-th variable of $\rho$ and the $r_i$-th variable of $\theta$ have the same sort $s_i$.
\end{itemize}
Equivalently, this means that
\begin{equation*}
    \rho = \begin{pmatrix}
            \text{$r_1$-th row of $\theta$}\\
            \text{$r_2$-th row of $\theta$}\\
            \vdots\\
            \text{$r_m$-th row of $\theta$}
    \end{pmatrix}\,, \tag{$\ast$}
\end{equation*}
which can also be formally expressed by the following pp formula:
\begin{multline*}
    \rho(x_1,\dots,\,x_m) =\\= \exists y_1\dots \exists y_{kl^n}\, \theta(y_1,\dots,\,y_{kl^n}) \land (x_1 = y_{r_1}) \land \dots \land (x_m = y_{r_m})\,.
\end{multline*}
We now claim that $\rho$ is defined by the following pp formula:
\begin{multline*}
    \rho(x_1,\dots,\,x_m) =\\= \exists y_1\dots \exists y_{kl^n}\, \widetilde{\Pol S}_n(y_1,\dots,\,y_{kl^n}) \land (x_1 = y_{r_1}) \land \dots \land (x_m = y_{r_m})\,.
\end{multline*}
The inclusion $\subseteq$ follows immediately from the definition of $\theta$ since $\Pol S$ is a clone and thus contains all projections. For the converse inclusion, let us consider a tuple
\begin{equation*}
    \boldsymbol{a}\in \exists y_1\dots \exists y_{kl^n}\, \widetilde{\Pol S}_n(y_1,\dots,\,y_{kl^n}) \land (x_1 = y_{r_1}) \land \dots \land (x_m = y_{r_m})\,.
\end{equation*}
By Remark~\ref{def:operationstuples}, $\widetilde{\Pol S}_n$ consists of precisely the $n$-ary $k$-operations from $\Pol S$ evaluated on the rows of $\theta$. Thus, there exists an $n$-ary $k$-operation $\boldsymbol{f} \in \Pol S$ such that
\begin{equation*}
    \boldsymbol{a} = \begin{pmatrix}
            \boldsymbol{f}^{(s_1)}(\text{$r_1$-th row of $\theta$})\\
            \boldsymbol{f}^{(s_2)}(\text{$r_2$-th row of $\theta$})\\
            \vdots\\
            \boldsymbol{f}^{(s_m)}(\text{$r_m$-th row of $\theta$})
    \end{pmatrix}\,.
\end{equation*}
Finally, since $\rho$ is preserved by all $k$-operations from $\Pol S$, we get that $\boldsymbol{a}\in\nobreak\rho$.
\end{proof}
\begin{corollary}\label{thm:invpol}
Let $S\subseteq R_A^k$. Then
\begin{equation*}
    \Inv\Pol(S) = \pp{S}\,.
\end{equation*}
\end{corollary}
\begin{proof}
The inclusion $\supseteq$ follows from Lemma~\ref{thm:invisrelclo}. The converse inclusion then follows from lemmata~\ref{thm:invpolstep1} and \ref{thm:invpolstep2}.
\end{proof}
Together with Theorem~\ref{thm:polinv}, we have proved the $\Pol$-$\Inv$ Galois connection for the multi-sorted clones. It remains to extend this result for the quantified relational clones.
\begin{lemma}\label{thm:leastrel}
Let $S\subseteq R_A^k$ and let $\rho\in R_A^k$. Then
\begin{equation*}
    \bigcap_{\substack{\sigma\,\in\,\pp{S}\\ \rho\,\subseteq\,\sigma}} \sigma = \{\boldsymbol{f}(\boldsymbol{b}_1,\dots,\,\boldsymbol{b}_n)\mid \boldsymbol{b}_1,\dots,\,\boldsymbol{b}_n\in\rho,\,\boldsymbol{f}\in\Pol S\}\,.
\end{equation*}
\end{lemma}
\begin{proof}
By definition, the intersection on the left-hand side is the least relation in $\pp{S}$ that contains $\rho$. To show that the right-hand side also satisfies this property, we denote
\begin{equation*}
    \tau = \{\boldsymbol{f}(\boldsymbol{b}_1,\dots,\,\boldsymbol{b}_n) \mid \boldsymbol{b}_1,\dots,\,\boldsymbol{b}_n\in\rho,\,\boldsymbol{f}\in\Pol S\}
\end{equation*}
and notice the following.
\begin{itemize}
    \item Since $\Pol S$ contains projections, we immediately get that $\rho\subseteq\tau$.
    \item $\tau$ is preserved by every operation in $\Pol S$. To see this, let $\boldsymbol{f} \in \Pol S$ be an $r$-ary operation, and take $\boldsymbol{c}_1,\dots,\,\boldsymbol{c}_r \in \tau$. We need to show that $\boldsymbol{f}(\boldsymbol{c}_1,\dots,\,\boldsymbol{c}_r) \in \tau$.

    By the definition of $\tau$, for each $\boldsymbol{c}_i$, there exist some $q_i$-ary operation $\boldsymbol{g}_i \in \Pol S$ and tuples $\boldsymbol{b}_{i,1},\dots,\,\boldsymbol{b}_{i,q_i} \in \rho$ such that
    \begin{equation*}
        \boldsymbol{c}_i = \boldsymbol{g}_i(\boldsymbol{b}_{i,1},\dots,\,\boldsymbol{b}_{i,q_i})\,.
    \end{equation*}
    Since $\rho$ is finite, say $\rho = \{\boldsymbol{a}_1,\dots,\,\boldsymbol{a}_m\}$, for each $\boldsymbol{g}_i$ there exists an $m$-ary operation $\boldsymbol{g}_i' \in \Pol S$ such that
    \begin{equation*}
        \boldsymbol{c}_i = \boldsymbol{g}_i'(\boldsymbol{a}_1,\dots,\,\boldsymbol{a}_m)\,.
    \end{equation*}
    This follows from the fact that $\Pol S$ is closed under composition: the operation $\boldsymbol{g}_i'$ is simply the composition of $\boldsymbol{g}_i$ with $m$-ary projections onto appropriate coordinates.

    Thus, we have
    \begin{equation*}
        \boldsymbol{f}(\boldsymbol{c}_1,\dots,\,\boldsymbol{c}_r) = \boldsymbol{f}(\boldsymbol{g}_1'(\boldsymbol{a}_1,\dots,\,\boldsymbol{a}_m),\dots,\,\boldsymbol{g}_r'(\boldsymbol{a}_1,\dots,\,\boldsymbol{a}_m)) = \boldsymbol{h}(\boldsymbol{a}_1,\dots,\,\boldsymbol{a}_m)\,,
    \end{equation*}
    where $\boldsymbol{h} \coloneqq \boldsymbol{f}(\boldsymbol{g}_1',\dots,\,\boldsymbol{g}_r') \in \Pol S$ by closedness under composition. Thus, by the definition of $\tau$, we have $\boldsymbol{h}(\boldsymbol{a}_1,\dots,\,\boldsymbol{a}_m) \in \tau$. By Corollary~\ref{thm:invpol}, it follows that $\tau \in \Inv\Pol S = \pp{S}$.
    \item Assume there is $\tau'\in\pp{S}$ such that $\rho\subseteq\tau'$. By Corollary~\ref{thm:invpol}, we have $\pp{S} = \Inv\Pol S$, so applying any $\boldsymbol{f} \in \Pol S$ to elements of $\tau'$ yields another element of $\tau'$. In particular, for every $\boldsymbol{b}_1,\dots,\,\boldsymbol{b}_n \in \rho\subseteq\tau'$ we have $\boldsymbol{f}(\boldsymbol{b}_1, \dots, \boldsymbol{b}_n)\in\tau'$, so we conclude that $\tau \subseteq \tau'$.
\end{itemize}
Therefore, $\tau$ is the least relation in $\pp{S}$ containing $\rho$, as required.
\end{proof}
The following theorem is a direct generalization of Proposition 3.15 in~\cite{Borner2009}.
\begin{theorem}\label{thm:invspol}
Let $S\subseteq R_A^k$. Then
\begin{equation*}
    \Inv\sPol(S) = \qpp{S}\,.
\end{equation*}
\end{theorem}
\begin{proof}
The inclusion $\supseteq$ is a direct corollary of Lemma~\ref{thm:invisrelclo}. It remains to prove the converse inclusion.

Consider an $n$-ary relation $\tau\in \Inv\sPol(S)$ whose $i$-th variable has a sort $s_i\in E_k$. Define
\begin{equation*}
    \rho(x_1^{s_1},\dots,\,x_n^{s_n},\,y_1^1,\dots,\,y_l^1,\dots,\,y_1^k,\dots,\,y_l^k) \coloneqq \tau(x_1^{s_1},\dots,\,x_n^{s_n})\,.
\end{equation*}
As it is not immediately clear that $\tau,\,\rho \in \qpp{S}$, we proceed by defining $\rho'$ as the least relation in $\pp{S}$ containing $\rho$. Equivalently, we put
\begin{equation*}
    \rho' = \bigcap_{\substack{\sigma\,\in\,\pp{S}\\ \rho\,\subseteq\,\sigma}} \sigma \stackrel{\ref{thm:leastrel}}{=} \{\boldsymbol{f}(\boldsymbol{b}_1,\dots,\,\boldsymbol{b}_m) \mid \boldsymbol{b}_1,\dots,\,\boldsymbol{b}_m\in\rho,\,\boldsymbol{f}\in\Pol S\}\,.
\end{equation*}
Finally, we define
\begin{multline*}
    \tau'(x_1^{s_1},\dots,\,x_n^{s_n}) = \forall y_1^1,\dots,\forall y_l^1\dots\forall y_1^k\dots\forall y_l^k\\ \rho'(x_1^{s_1},\dots,\,x_n^{s_n},\,y_1^1,\dots,\,y_l^1,\dots,\,y_1^k,\dots,\,y_l^k)\,.
\end{multline*}
Notice that $\tau'$ is qpp-definable from relations in $S$ and that $\tau\subseteq \tau'$. To finish the proof of $\tau\in\qpp{S}$, we prove the converse inclusion, i.e. $\tau'\subseteq \tau$. Let us take $\boldsymbol{a}\in\tau'$.

By the definition of $\tau'$, the tuple $(\boldsymbol{a}^{(1)},\dots,\,\boldsymbol{a}^{(n)},\,b_1^1,\dots,\,b_l^k)$ is in $\rho'$ for all $b_i^j\in A$, where $i\in E_l$ and $j\in E_k$. In particular, we can choose the constants $b_i^j$ such that
\begin{equation*}
    \{b_1^1,\dots,\,b_l^1\} = \{b_1^2,\dots,\,b_l^2\} = \dots = \{b_1^k,\dots,\,b_l^k\} = A\,.
\end{equation*}
By the definition of $\rho'$ and by Lemma~\ref{thm:leastrel}, there is $m$-ary $\boldsymbol{f}\in\Pol\qpp{S}$, tuples $\boldsymbol{a}_1,\dots,\,\boldsymbol{a}_m\in\tau$, and constants $b_{i,h}^j\in A$ such that
\begin{equation*}
    \begin{pmatrix}
        \boldsymbol{a}^{(1)}\\
        \vdots\\
        \boldsymbol{a}^{(n)}\\
        b_1^1\\
        \vdots\\
        b_l^1\\
        \vdots\\
        b_1^k\\
        \vdots\\
        b_l^k
    \end{pmatrix} = \begin{pmatrix}
        \boldsymbol{f}^{(s_1)}(\boldsymbol{a}_1^{(1)},\dots,\,\boldsymbol{a}_m^{(1)})\\
        \vdots\\
        \boldsymbol{f}^{(s_n)}(\boldsymbol{a}_1^{(n)},\dots,\,\boldsymbol{a}_m^{(n)})\\
        \boldsymbol{f}^{(1)}(b_{1,1}^1,\dots,\,b_{1,m}^1)\\
        \vdots\\
        \boldsymbol{f}^{(1)}(b_{l,1}^1,\dots,\,b_{l,m}^1)\\
        \vdots\\
        \boldsymbol{f}^{(k)}(b_{1,1}^k,\dots,\,b_{1,m}^k)\\
        \vdots\\
        \boldsymbol{f}^{(k)}(b_{l,1}^k,\dots,\,b_{l,m}^k)
    \end{pmatrix}\,.
\end{equation*}
Notice that by the choice of $b_i^j$, $\boldsymbol{f}$ is surjective, hence $\boldsymbol{f}\in\sPol S$. And since $\tau\in\Inv\sPol S$, we get
\begin{equation*}
    \boldsymbol{a} = \begin{pmatrix}
        \boldsymbol{f}^{(s_1)}(\boldsymbol{a}_1^{(1)},\dots,\,\boldsymbol{a}_m^{(1)})\\
        \vdots\\
        \boldsymbol{f}^{(s_n)}(\boldsymbol{a}_1^{(n)},\dots,\,\boldsymbol{a}_m^{(n)})
    \end{pmatrix} \in\tau\,.
\end{equation*}
Thus, we conclude that $\tau = \tau' \in \qpp{S}$, completing the proof.
\end{proof}
Finally, combining Observation~\ref{thm:galoisproperty}, Theorem~\ref{thm:polinv}, and Theorem~\ref{thm:invspol}, we conclude the proof of Theorem~\ref{thm:Galois}.
\section{Deferred proofs}\label{sec:defproofs}
This section presents all proofs that were previously omitted due to their technical nature. Each proof is preceded by a restatement of the corresponding result for clarity.
\subsection{Construction of canonical relations}\label{sec:defproofscanon}
Recall that $k\in\N$ is fixed.
\begin{THMgaussrearrangementTHM}
For every relation $\rho'\in\KR^k$, there are constants $m,\,n,\,l\in\N_0$, and $p_i,\,q_j,\,r_h\in E_k$, $a_{i,j},\,b_i,\,c_h\in\{0,\,1\}$ for all $i\in E_m$, $j\in E_n$, $h\in E_l$ such that the relation $\rho$ defined by
\begin{equation*}
    \rho\!\left(\begin{array}{@{\,}l@{\,}}
        x_1^{p_1},\dots,\,x_m^{p_m}, \\
        y_1^{q_1},\dots,\,y_n^{q_n}, \\
        z_1^{r_1},\dots,\,z_l^{r_l}
    \end{array}\right) = \disj{\begin{array}{r@{\mskip\medmuskip}c@{\mskip\medmuskip}l}
        x_1^{p_1} & = & a_{1,1}y_1^{q_1} + \dots + a_{1,n}y_n^{q_n} + b_1 \\
        &\vdots&\\
        x_m^{p_m} & = & a_{m,1}y_1^{q_1} + \dots + a_{m,n}y_n^{q_n} + b_m\\
        z_1^{r_1} & = & c_1 \\
        & \vdots & \\
        z_l^{r_l} & = & c_l
    \end{array}}
\end{equation*}
is similar to $\rho'$.
\end{THMgaussrearrangementTHM}
\begin{proof}
The result follows directly from De~Morgan's laws and Gauss-Jordan elimination. Given a general form of $\rho'$, we observe that
\begin{align*}
    \rho'(u_1,\dots,\,u_\mu) &= \bigvee_{i\,=\,1}^{\nu} (\alpha_{i,1}u_1 + \alpha_{i,2}u_2 + \dots + \alpha_{i,\mu}u_\mu = \beta_i) = \\
    &= \lnot\bigwedge_{i\,=\,1}^{\nu} (\alpha_{i,1}u_1 + \alpha_{i,2}u_2 + \dots + \alpha_{i,\mu}u_\mu = \beta_i + 1)\,.
\end{align*}
Thus, every relation in $\KR^k$ is a~complement of an~affine space. The desired relation $\rho$ is obtained by performing Gauss-Jordan elimination on this system, followed by a~suitable rearrangement of variables.
\end{proof}
Proving the lemmata~\ref{thm:decomposition}, \ref{thm:canon45}, \ref{thm:canon2}, \ref{thm:canon3}, \ref{thm:disjeq01}, and \ref{thm:canon12} reduces to establishing that $\qpp{S} = \qpp{S'}$ for certain sets $S,\, S' \subseteq R_{\{0,1\}}^k$. The common strategy in all these proofs is to construct each relation in $S$ as a qpp formula using relations from $S'$ and, conversely, to express each relation in $S'$ using relations from $S$.

To keep the presentation concise, we do not explicitly write out formulas whenever the resulting relation is obtained by a direct application of elementary operations (eo1)--(eo5); see Definition~\ref{def:eo} and Remark~\ref{rem:applyingeo} for the details.
\begin{THMdecompositionTHM}
Suppose
\begin{align*}
    \rho\!\left(\begin{array}{@{\,}l@{\,}}
        x_1^{p_1},\dots,\,x_m^{p_m}, \\
        y_1^{q_1},\dots,\,y_n^{q_n}, \\
        z_1^{r_1},\dots,\,z_l^{r_l}
    \end{array}\right) &= \disj{\begin{array}{r@{\mskip\medmuskip}c@{\mskip\medmuskip}l}
        x_1^{p_1} & = & a_{1,1}y_1^{q_1} + \dots + a_{1,n}y_n^{q_n} + b_1 \\
        &\vdots&\\
        x_m^{p_m} & = & a_{m,1}y_1^{q_1} + \dots + a_{m,n}y_n^{q_n} + b_m\\
        z_1^{r_1} & = & c_1 \\
        & \vdots & \\
        z_l^{r_l} & = & c_l
    \end{array}}\,,\\
    \sigma\!\left(\begin{array}{@{\,}l@{\,}}
        u_{1}^{p_1},\dots,\,u_{m}^{p_m}, \\
        v_{1}^{p_1},\dots,\,v_{m}^{p_m}, \\
        z_1^{r_1},\dots,\,z_l^{r_l}
    \end{array}\right)&= \disj{\begin{array}{r@{\mskip\medmuskip}c@{\mskip\medmuskip}l}
        u_{1}^{p_1} & = & v_{1}^{p_1} \\
        &\vdots&\\
        u_{m}^{p_m} & = & v_{m}^{p_m} \\
         z_1^{r_1} & = & c_1 \\
        & \vdots & \\
        z_l^{r_l} & = & c_l
    \end{array}}\,,\\
    \lambda_1(x_1^{p_1},\,y_1^{q_1},\dots,\,y_n^{q_n}) &= (x_1^{p_1} = a_{1,1}y_1^{q_1} + \dots + a_{1,n}y_n^{q_n} + b_1)\,,\\
    &\hspace*{0.5em}\vdots\\
    \lambda_m(x_m^{p_m},\,y_1^{q_1},\dots,\,y_n^{q_n}) &= (x_m^{p_m} = a_{m,1}y_1^{q_1} + \dots + a_{m,n}y_n^{q_n} + b_m)\,,
\end{align*}
where $m,\,n,\,l\in\N_0$, and $p_i,\,q_j,\,r_h\in E_k$, $a_{i,j},\,b_i,\,c_h\in\{0,\,1\}$ for all $i\in E_m$, $j\in E_n$, $h\in E_l$. Then
\begin{enumerate}
    \item $\sigma\in\pp{\rho}$,
    \item $\qpp{\rho} = \qpp{\sigma,\,\lambda_1,\dots,\,\lambda_m}$.
\end{enumerate}
\end{THMdecompositionTHM}
\begin{proof}
For statement (1), we verify the following pp formula:
\begin{multline*}
    \sigma\!\left(\begin{array}{@{\,}l@{\,}}
        x_{1,1}^{p_1},\dots,\,x_{m,1}^{p_m}, \\
        x_{1,2}^{p_1},\dots,\,x_{m,2}^{p_m}, \\
        z_1^{r_1},\dots,\,z_l^{r_l}
    \end{array}\right) =\\= \exists y_1^{q_1},\dots,\,\exists y_n^{q_n} \,\bigwedge_{(d_1,\dots,\,d_m)\,\in\,\{1,\,2\}^m} \rho\!\left(\begin{array}{@{\,}l@{\,}}
        x_{1,d_1}^{p_1},\dots,\,x_{m,d_m}^{p_m}, \\
        y_1^{q_1},\dots,\,y_n^{q_n}, \\
        z_1^{r_1},\dots,\,z_l^{r_l}
    \end{array}\right)\,.
\end{multline*}
The equality is immediate when $z_h^{r_h} = c_h$ for some $h\in E_l$. Otherwise, if $z_h^{r_h} \neq c_h$ for all $h\in E_l$, we distinguish two cases:
\begin{itemize}
    \item If $x_{i,1}^{p_i}$ and $x_{i,2}^{p_i}$ take the same value for some $i\in E_m$, we can choose witnesses such that
    \begin{equation*}
        x_{i,1}^{p_i} = x_{i,2}^{p_i} = a_{i,1}y_1^{q_1} + \dots + a_{i,n}y_n^{q_n} + b_i\,.
    \end{equation*}
    \item If $x_{i,1}^{p_i}$ and $x_{i,2}^{p_i}$ take different values for all $i\in E_m$, then any choice of witnesses satisfies exactly one of the following equations for each $i\in E_m$:
    \begin{align*}
        x_{i,1}^{p_i} &= a_{i,1}y_1^{q_1} + \dots + a_{i,n}y_n^{q_n} + b_i\,,\\
        x_{i,2}^{p_i} &= a_{i,1}y_1^{q_1} + \dots + a_{i,n}y_n^{q_n} + b_i\,,
    \end{align*}
    This means that there exists a tuple $(e_1,\dots,\,e_m)\in\{1,\,2\}^m$ such that the term
    \begin{equation*}
        \rho\!\left(\begin{array}{@{\,}l@{\,}}
            x_{1,e_1}^{p_1},\dots,\,x_{m,e_m}^{p_m}, \\
            y_1^{q_1},\dots,\,y_n^{q_n}, \\
            z_1^{r_1},\dots,\,z_l^{r_l}
        \end{array}\right)
    \end{equation*}
    is not witnessed.
\end{itemize}
This concludes the proof of statement (1). To prove statement (2), we observe:
\begin{itemize}
    \item Each $\lambda_j$ belongs to $\qpp{\rho}$, as it can be obtained by universally quantifying all variables $z_h^{r_h}$ for $h\in E_l$ and all variables $x_i^{p_i}$ for $i\in E_m\setminus\{j\}$ in $\rho$.
    \item $\rho\in\qpp{\sigma,\,\lambda_1,\dots,\,\lambda_m}$ as $\rho$ is obtained by a sequence of compositions (eo4) of $\sigma$ with $\lambda_1,\dots,\,\lambda_m$.\qedhere
\end{itemize}
\end{proof}
\begin{THMcanon45THM}
Let
\begin{multline*}
    \rho(x_1^1,\dots,\,x_{l_1}^1,\,x_1^2,\dots,\,x_{l_2}^2,\dots,\,x_1^k,\dots,\,x_{l_k}^k) =\\= (x_1^1 + \dots + x_{l_1}^1 + x_{1}^2 + \dots + x_{l_2}^2+\dots+x_{1}^k + \dots + x_{l_k}^k = b)\,,
\end{multline*}
where $l_1,\dots,\,l_k\in\N_0$ and $l_1 + \dots + l_k \geq 3$. We define
\begin{align*}
    I &= \{i\in E_k\mid l_i\neq 0\}\,, \qquad \text{(sorts with at least one variable)}\\
    O &= \{i\in E_k\mid l_i\text{ is odd}\}\,. \quad \text{(sorts with an odd number of variables)}
\end{align*}
Then:
\begin{enumerate}
    \item If $O = \emptyset$ and $b = 0$, then
    \begin{equation*}
        \qpp{\rho} = \qpp{\{x^i+y^i = u^j + v^j \mid i,\,j\in I\}}\,.
    \end{equation*}
    \item If $O = \emptyset$ and $b=1$, then
    \begin{equation*}
        \qpp{\rho} = \qpp{\{(x^i+y^i = u^j + v^j),\,(x^i + y^i = 1) \mid i,\,j\in I\}}\,.
    \end{equation*}
    \item If $O \neq \emptyset$, letting $\{p_1,\dots,\,p_m\} = O$, we have
    \begin{equation*}
        \qpp{\rho} = \qpp{\{x^{p_1} +\dots + x^{p_m} =  b\}\cup\{x^i+y^i = u^j + v^j \mid i,\,j\in I\}}\,.
    \end{equation*}
\end{enumerate}
\end{THMcanon45THM}
\begin{proof}
The inclusion ``$\supseteq$'' is established in each case by constructing the necessary relations through identification of variables (eo3) and removal of dummy variables (eo1), sometimes preceded by a suitable sequence of compositions (eo4) of $\rho$ with itself.

For the inclusion ``$\subseteq$'', the required relations are obtained directly through a sequence of compositions (eo4).
\end{proof}
\begin{THMcanon2THM}
Let $p_1,\dots,\,p_m,\,r_1,\dots,\,r_l\in E_k$ and $c_1,\dots,\,c_l\in\{0,\,1\}$. Then
\begin{equation*}
    \qpp{\disj{\begin{array}{r@{\mskip\medmuskip}c@{\mskip\medmuskip}l}
        u_{1}^{p_1} & = & v_{1}^{p_1} \\
        &\vdots&\\
        u_{m}^{p_m} & = & v_{m}^{p_m} \\
        z_1^{r_1} & = & c_1 \\
        z_2^{r_2} & = & c_2 \\
        & \vdots & \\
        z_l^{r_l} & = & c_l
    \end{array}}} = \qpp{\disj{\begin{array}{r@{\mskip\medmuskip}c@{\mskip\medmuskip}l}
        u_{1}^{p_1} & = & v_{1}^{p_1} \\
        &\vdots&\\
        u_{m}^{p_m} & = & v_{m}^{p_m} \\
        z_2^{r_2} & = & c_2 \\
        & \vdots & \\
        z_{l}^{r_{l}} & = & c_{l}
    \end{array}},\,\disj{\begin{array}{r@{\mskip\medmuskip}c@{\mskip\medmuskip}l}
        u_{1}^{p_1} & = & v_{1}^{p_1} \\
        z_1^{r_1} & = & c_1
    \end{array}}}\,.
\end{equation*}
\end{THMcanon2THM}
\begin{proof}
To establish the inclusion ``$\supseteq$'', we construct the required relations by universal quantification of some variables (eo5). The relation required to establish inclusion ``$\subseteq$'' is constructed by composition (eo4).
\end{proof}
\begin{THMcanon3THM}
Let $p_1,\dots,\,p_m\in E_k$. Then
\begin{equation*}
    \qpp{\disj{\begin{array}{r@{\mskip\medmuskip}c@{\mskip\medmuskip}l}
        u_{1}^{p_1} & = & v_{1}^{p_1} \\
        u_{2}^{p_2} & = & v_{2}^{p_2} \\
        &\vdots&\\
        u_{m}^{p_m} & = & v_{m}^{p_m} 
    \end{array}}} = \qpp{\disj{\begin{array}{r@{\mskip\medmuskip}c@{\mskip\medmuskip}l}
        u_{2}^{p_2} & = & v_{2}^{p_2} \\
        &\vdots&\\
        u_{m}^{p_m} & = & v_{m}^{p_m}
    \end{array}},\,\disj{\begin{array}{r@{\mskip\medmuskip}c@{\mskip\medmuskip}l}
        u_{1}^{p_1} & = & v_{1}^{p_1} \\
        u_{2}^{p_2} & = & v_{2}^{p_2} 
    \end{array}}}\,.
\end{equation*}
Furthermore, if $i\in E_k$, then
\begin{equation*}
    \qpp{x^i = y^i\lor u^i = v^i} = \qpp{x^i = y^i \lor y^i = z^i}\,.
\end{equation*}
\end{THMcanon3THM}
\begin{proof}
The proof of the first equality is analogous to that of Lemma~\ref{thm:canon2}. The~``$\subseteq$'' direction of the second equality follows from Lemma~\ref{thm:decomposition}, applied to the relation $\rho(x_1^i,\,x_2^i,\,y^i) = (x_1^i = y^i \lor x_2^i = y^i)$, while the converse inclusion is immediate.
\end{proof}
\begin{THMdisjeq01THM}
Let $i,\,r_1,\dots,\,r_l\in E_k$ and $c_1,\dots,\,c_l\in\{0,\,1\}$. Then
\begin{equation*}
    \qpp{\disj{\begin{array}{r@{\mskip\medmuskip}c@{\mskip\medmuskip}l}
        x^i & = & 0 \\
        y^i & = & 1 \\
        z_1^{r_1} & = & c_1 \\
        z_2^{r_2} & = & c_2 \\
        & \vdots & \\
        z_l^{r_l} & = & c_l
    \end{array}}} = \qpp{\disj{\begin{array}{r@{\mskip\medmuskip}c@{\mskip\medmuskip}l}
        x^i & = & 0 \\
        y^i & = & 1 \\
        z_2^{r_2} & = & c_2 \\
        & \vdots & \\
        z_l^{r_l} & = & c_l
    \end{array}},\,\disj{\begin{array}{r@{\mskip\medmuskip}c@{\mskip\medmuskip}l}
        x^i & = & 0 \\
        y^i & = & 1 \\
        z_1^{r_1} & = & c_1 \\
    \end{array}}}\,.
\end{equation*}
\end{THMdisjeq01THM}
\begin{proof}
The proof is analogous to that of Lemma~\ref{thm:canon2}.
\end{proof}
\begin{THMcanon12THM}
Let $i,\,j\in E_k$ and $b\in\{0,\,1\}$. Then
\begin{equation*}
    \qpp{\disj{\begin{array}{r@{\mskip\medmuskip}c@{\mskip\medmuskip}l}
        x^i & = & 0 \\
        y^i & = & 1 \\
        z^j & = & b \\
    \end{array}}} = \qpp{\disj{\begin{array}{r@{\mskip\medmuskip}c@{\mskip\medmuskip}l}
        x^i & = & y^i \\
        z^j & = & b \\
    \end{array}}, \disj{\begin{array}{r@{\mskip\medmuskip}c@{\mskip\medmuskip}l}
        x^i & = & 0 \\
        y^i & = & 1 \\
    \end{array}}}\,.
\end{equation*}
\end{THMcanon12THM}
\begin{proof}
The relation required for the inclusion~``$\subseteq$'' is obtained by composition (eo4). For the converse inclusion, we note that one of the relations is obtained via universal quantification (eo5), while the other is obtained by conjunction (eo6) as:
\begin{equation*}
    \disj{\begin{array}{r@{\mskip\medmuskip}c@{\mskip\medmuskip}l}
        x^i + y^i & = & 0 \\
        z^j & = & b \\
    \end{array}} = \disj{\begin{array}{r@{\mskip\medmuskip}c@{\mskip\medmuskip}l}
        x^i & = & 0 \\
        y^i & = & 1 \\
        z^j & = & b \\
    \end{array}} \land \disj{\begin{array}{r@{\mskip\medmuskip}c@{\mskip\medmuskip}l}
        y^i & = & 0 \\
        x^i & = & 1 \\
        z^j & = & b \\
    \end{array}}\,.
\end{equation*}
\end{proof}

Since all the necessary statements for Theorem~\ref{thm:CR} have been established, it only remains to prove Theorem~\ref{thm:propCR}. The key idea is that this theorem follows almost directly from the classification of quantified relational clones generated by single canonical relations—that is, the sets $\qpp{\rho}$ for every $\rho\in\CR^k$. We present this classification in Table~\ref{tab:qppcanon} and formally establish it in Theorem~\ref{thm:qppcanon}.

For conciseness, the elements of each quantified relational clone are listed up to permutations of variables and the presence of dummy variables. Additionally, we omit explicit mention of ``trivial'' relations, meaning those in $\qpp{\emptyset}$. In summary, for each $\rho\in\CR^k$, we describe only the set $S_\rho$ such that
\begin{equation*} 
    \qpp{\rho}\cap\KR^k = \eo{S_\rho}{2}\cup\qpp{\emptyset}\,. \end{equation*}
This description is sufficient to fully determine $\qpp{\rho}$.

\newcolumntype{C}{!{\vrule}c!{\vrule}}
\newcolumntype{R}{c!{\vrule}}
\NiceMatrixOptions{cell-space-limits = 4pt}
\begin{table}[!htb]
    \centering
    \resizebox{\textwidth}{!}{
    \begin{NiceTabular}{CRR}\hline
        Type & $\rho$ & $S_\rho$ \\ \hline\hline
        (c1) & $x^i = 0 \lor y^i = 1$ & $\{x^i = 0 \lor y^i = 1\}$ \\ \hline\hline
        \Block{2-1}{(c2)} & \begin{tabular}{c}$x^i = y^i \lor u^j = b$\\ where $i\neq j$\end{tabular} & $\left\{ \disj{\begin{array}{r@{\mskip\medmuskip}c@{\mskip\medmuskip}l}
        x^i & = & y^i \\
        z_1^j & = & b \\
        &\vdots &\\
        z_n^j & = & b
    \end{array}},\, \disj{\begin{array}{r@{\mskip\medmuskip}c@{\mskip\medmuskip}l}
        z_1^j & = & b \\
        &\vdots &\\
        z_n^j & = & b
    \end{array}}~\middle|~ n\in\N \right\}$ \\ \cline{2-3}
         & $x^i = y^i \lor u^i = b$ & $\left\{ \disj{\begin{array}{r@{\mskip\medmuskip}c@{\mskip\medmuskip}l}
        x^i & = & y^i \\
        z_1^i & = & b \\
        &\vdots &\\
        z_n^i & = & b
    \end{array}},\, \disj{\begin{array}{r@{\mskip\medmuskip}c@{\mskip\medmuskip}l}
        z_1^i & = & b \\
        z_2^i & = & b\\
        &\vdots &\\
        z_n^i & = & b
    \end{array}},\, \disj{\begin{array}{r@{\mskip\medmuskip}c@{\mskip\medmuskip}l}
        z_1^i & = & b+1 \\
        z_2^i & = & b\\
        &\vdots &\\
        z_n^i & = & b
    \end{array}}~\middle|~ n\in\N \right\}$ \\ \hline\hline
        \Block{2-1}{(c3)} & \begin{tabular}{c}$x^i = y^i \lor u^j = v^j$\\ where $i\neq j$\end{tabular} & $\left\{\disj{\begin{array}{r@{\mskip\medmuskip}c@{\mskip\medmuskip}l}
        a_{1,1}x_1^i + \dots + a_{1,m}x_m^i & = & b_1\\
        &\vdots &\\
        a_{n,1}x_1^i + \dots + a_{n,m}x_m^i & = & b_n\\
        \alpha_{1,1}x_1^j + \dots + \alpha_{1,\mu}x_\mu^j & = & \beta_1\\
        &\vdots &\\
        \alpha_{\nu,1}x_1^j + \dots + \alpha_{\nu,\mu}x_\mu^j & = & \beta_\nu
    \end{array}}~\middle|~\begin{array}{c}
    m,\,\mu,\,n,\,\nu\in\N_0\\
    \text{and}\\
    a_{p,q},\,b_p\in\{0,\,1\},\\
    \alpha_{p,q},\,\beta_p\in\{0,\,1\}\\
    \text{such that}\\
    a_{p,1} + \dots + a_{p,m} = 0\\
    \alpha_{p,1} + \dots + \alpha_{p,\mu} = 0
    \\\end{array}\right\}$ \\ \cline{2-3}
         & $x^i = y^i\lor y^i = z^i$ & $\left\{\disj{\begin{array}{r@{\mskip\medmuskip}c@{\mskip\medmuskip}l}
        a_{1,1}x_1^i + \dots + a_{1,m}x_m^i & = & b_1 \\
        &\vdots &\\
        a_{n,1}x_1^i + \dots + a_{n,m}x_m^i & = & b_n
    \end{array}}~\middle|~\begin{array}{c}
    m,\,n\in\N\text{ and}\\
    a_{p,q},\,b_p\in\{0,\,1\}\\
    \text{such that}\\
    a_{p,1} + \dots + a_{p,m} = 0
    \\\end{array}\right\}$ \\ \hline\hline
        (c4) & $x^i + y^i = 1$ & $\{x^i + y^i = 1\}$\\ \hline\hline
        (c5) & $x^i + y^i = u^j + v^j$ & $\{x_1^i+\dots+x_{2n}^i + y_1^j+\dots+y_{2m}^j = 0~|~n,\,m\in\N \}$ \\ \hline\hline
        \Block{2-1}{(c6)} & \begin{tabular}{c}$x^{s_1} + x^{s_2} + \dots + x^{s_n} = b$\\ where $n\geq 3$\end{tabular} & $\left\{\begin{array}{c}\sum_{q\,=\,1}^n\sum_{p\,=\,1}^{2m_q+1} x_p^{s_q} = b,\\\sum_{q\,=\,1}^n\sum_{p\,=\,1}^{2m_q} x_p^{s_q} = 0 \end{array}~\middle|~ m_1,\dots,\,m_n\in\N_0\right\}$\\ \cline{2-3}
        & $x^{s_1} + x^{s_2} = b$ & $\{x^{s_1} + x^{s_2} = b\}$ \\ \hline\hline
        (c7) & $\disj{\begin{array}{r@{\mskip\medmuskip}c@{\mskip\medmuskip}l}
        x_1^{s_1} = b_1 \,\lor & \cdots & \lor\, x_{m_1}^{s_1} = b_1  \\
        x_1^{s_2} = b_2 \,\lor & \cdots & \lor\, x_{m_2}^{s_2} = b_2 \\
        &\vdots& \\
        x_1^{s_n} = b_n \,\lor & \cdots & \lor\, x_{m_n}^{s_n} = b_n
    \end{array}}$ & $\left\{\disj{\begin{array}{r@{\mskip\medmuskip}c@{\mskip\medmuskip}l}
        x_1^{s_1} = b_1 \,\lor & \cdots & \lor\, x_{l_1}^{s_1} = b_1  \\
        x_1^{s_2} = b_2 \,\lor & \cdots & \lor\, x_{l_2}^{s_2} = b_2 \\
        &\vdots& \\
        x_1^{s_n} = b_n \,\lor & \cdots & \lor\, x_{l_n}^{s_n} = b_n
    \end{array}}~\middle|~\begin{array}{c}\text{for each } p\in E_n\\l_p\in\N_0,\\l_p\leq m_p\end{array}\right\}$ \\ \hline
    \end{NiceTabular}
    }
    \caption{The sets $S_\rho$ with the property $\qpp{\rho}\cap\KR^k = \eo{S_\rho}{2}\cup\qpp{\emptyset}$ for each $\rho\in\CR^k$.}
    \label{tab:qppcanon}
\end{table}

\begin{theorem} \label{thm:qppcanon}
For each $\rho\in\CR^k$, the set $S_\rho$ presented in Table~\ref{tab:qppcanon} satisfies
\begin{equation*}
    \qpp{\rho}\cap\KR^k = \eo{S_\rho}{2}\cup\qpp{\emptyset}\,.
\end{equation*}
\end{theorem}
\begin{proof}
Verification follows directly from the tools developed in Section~\ref{sec:eo}. In particular:
\begin{itemize}
    \item Each relation in $S_\rho$ can be derived from $\rho$ by a direct application of elementary operations. In most cases, the construction is straightforward; the only somewhat nontrivial case arises for canonical relations of type (c3), where we use the identity:
    \begin{equation*}
        (x_1 + x_2 + x_3 + x_4 = 0) = \bigwedge_{i\,\in\,E_4}\bigvee_{\substack{j\,\in\,E_4\\ j\,\neq\,i}} (x_i = x_j)\,.
    \end{equation*}
    \item The closedness of the sets $\eo{S_\rho}{2}\cup\qpp{\emptyset}$ follows from the approach outlined in Remark~\ref{rem:closedness}. While this remark itself refers to some properties of canonical relations, it ultimately relies only on Theorem~\ref{thm:CR} and fundamental properties of elementary operations.\qedhere
\end{itemize}
\end{proof}

\begin{THMpropCRTHM}
If $\rho\in\CR^k$, then
\begin{enumerate}
    \item there is no $\rho'\in\qpp{\rho}$ of lower arity than $\rho$ such that $\qpp{\rho'} = \qpp{\rho}$;
    \item for all $n\in\N$ and $\sigma_1,\,\sigma_2,\dots,\,\sigma_n\in\qpp{\rho}$ such that
    \begin{equation*}
        \forall i\in E_n\quad \qpp{\sigma_i} \subsetneq \qpp{\rho}\,,
    \end{equation*}
    we have $\qpp{\sigma_1,\,\sigma_2,\dots,\,\sigma_n}\subsetneq\qpp{\rho}$\,.
\end{enumerate}
\end{THMpropCRTHM}
\begin{proof}
Statement (1) follows from the classification given in Table~\ref{tab:qppcanon}. Specifically, for each $\sigma\in S_\rho$ with lower arity than $\rho$, we observe that
\begin{equation*}
    \qpp{\sigma} \subsetneq \qpp{\rho}\,.
\end{equation*}
This is immediate, as all such relations $\sigma$ are simple and there are only a~few cases to consider.

For statement (2), we compute the set
\begin{equation*}
    T_\rho \coloneqq \{\sigma\in\qpp{\rho} \mid \qpp{\sigma} \subsetneq \qpp{\rho} \}
\end{equation*}
for each $\rho\in\CR^k$. Then, we observe that these sets have the property
\begin{equation*}
    \qpp{T_\rho} \subsetneq \qpp{\rho}\,, \tag{$\ast$}
\end{equation*}
which immediately establishes the claim.

Both computation of $T_\rho$ and verification of $(\ast)$ are straightforward applications of elementary operations and Remark~\ref{rem:closedness}.

The only cases worth highlighting are when $\rho$ is of type (c3) or (c6). In the former, the computation of $T_\rho$ simplifies using Lemma~\ref{thm:decomposition}. In the latter, when $\rho$ is of the form
\begin{equation*}
    \rho(x^{s_1},\, x^{s_2},\dots,\, x^{s_n}) = (x^{s_1} + x^{s_2} + \dots + x^{s_n} = b)\,,
\end{equation*}
the case for $n=2$ is trivial, while for $n\geq 3$, we obtain
\begin{equation*}
    T_\rho = \eo{\left\{\sum_{q\,=\,1}^n\sum_{p\,=\,1}^{2m_q} x_p^{s_q} = 0 ~\middle|~ m_1,\dots,\,m_n\in\N_0\right\}}{2}\cup\qpp{\emptyset}\,,
\end{equation*}
which is a quantified relational clone strictly contained in $\qpp{\rho}$, completing the~proof.
\end{proof}
\subsection{Auxiliary statements from order theory}\label{sec:defproofslattice}
\begin{THMACCFGTHM}
Let $A$ be a~nonempty countable set and let $\cls$ be an algebraic closure operator on $A$. The poset $(\mathfrak{L}_{\cls}(A),\,\subseteq)$ satisfies ACC if and only if every element of $\mathfrak{L}_{\cls}(A)$ is finitely generated with respect to $\cls$.
\end{THMACCFGTHM}
\begin{proof}
($\implies$) Let us assume that $(\mathfrak{L}_{\cls}(A),\,\subseteq)$ satisfies ACC. Let us pick $B\in \mathfrak{L}_{\cls}(A)$; we will show that $B$ is finitely generated.

Since $A$ is countable, we may put $B = \{a_1,\,a_2,\,a_3,\dots\}$. We then assume the~following infinite chain in $(\mathfrak{L}_{\cls}(A),\,\subseteq)$:
\begin{equation*}
    \cls\{a_1\}\subseteq \cls\{a_1,\,a_2\}\subseteq\cls\{a_1,\,a_2,\,a_3\} \subseteq \dots \subseteq B\,.
\end{equation*}
Since $\mathfrak{L}_{\cls}(A)$ satisfies ACC, there is $n\in\N$ such that
\begin{equation*}
    \cls\{a_1,\dots,\,a_n\} = \cls\{a_1,\dots,\,a_n,\,a_{n+1}\} = \cls\{a_1,\dots,\,a_n,\,a_{n+1},\,a_{n+2}\} = \dots
\end{equation*}
Thus
\begin{equation*}
    \bigcup_{l\,\in\, n} \cls\{a_1,\dots,\,a_l\} = \cls\{a_1,\dots,\,a_n\}\,,
\end{equation*}
and since
\begin{equation*}
    B = \bigcup_{l\,\in\, \N} \{a_1,\dots,\,a_l\} \subseteq \bigcup_{l\,\in\, \N} \cls\{a_1,\dots,\,a_l\} \subseteq \bigcup_{\substack{C\,\subseteq\,B\\ |C|\,\leq\,\infty}} \cls C = \cls B = B\,,
\end{equation*}
we have shown that $B$ is finitely generated by the set $\{a_1,\dots,\,a_n\}$.

($\impliedby$) Let us assume that every element of $(\mathfrak{L}_{\cls}(A),\,\subseteq)$ is finitely generated with respect to $\cls$. Let us consider an ascending chain
\begin{equation*}
    B_1 \subseteq B_2 \subseteq B_3 \subseteq \dots\,,
\end{equation*}
where $B_1,\,B_2,\,B_3,\dots\in \mathfrak{L}_{\cls}(A)$. In order to show that this chain terminates, we consider the set
\begin{equation*}
    B = \bigcup_{l\,\in\,\N} B_l\,.
\end{equation*}
We may notice that $B = \cls B$. Indeed, for any $a\in \cls B$, we have
\begin{equation*}
    a\in \cls B = \bigcup_{\substack{C\,\subseteq\,B\\ |C|\,\leq\,\infty}} \cls C\,,
\end{equation*}
thus, there is a finite set $\{b_1,\dots,\,b_n\}\subseteq B$ such that $a\in \cls \{b_1,\dots,\,b_n\}$. For each $i\in E_n$, the element $b_i$ is contained within some $B_{q_i}$; thus, for $q = \max \{q_1,\dots,\,q_n\}$, we have
\begin{equation*}
    a\in \cls \{b_1,\dots,\,b_n\} \subseteq \cls B_q = B_q\subseteq B\,.
\end{equation*}
Hence, $B$ is a closed set, which is finitely generated by the assumption. Let us put $B = \cls\{a_1,\dots,\,a_m\}$ for some $m\in\N$. Then, for each $i\in E_m$, the element $a_i$ is contained within some $B_{p_i}$; thus, for $p = \max\{p_1,\dots,\,p_m\}$, we have $\{a_1,\dots,\,a_m\} \subseteq B_p$, which implies
\begin{equation*}
     B = \cls \{a_1,\dots,\,a_m\} \subseteq \cls B_p = B_p\subseteq B_{p+1} \subseteq B_{p+2} \subseteq \dots \subseteq B\,.
\end{equation*}
Thus, all the terms must be equal, which finishes the proof of ACC.
\end{proof}

\begin{definition}[Order theory concepts III]
Let $(P,\,\leq)$ be a~poset. A~set $O\subseteq P$ such that $\forall\,p,\,q\in O: (p\nleq q)\land(q\nleq p)$ is called \textit{antichain}.

Let $Q\subseteq P$. An element $p\in Q$ is called \textit{minimal} in $Q$ if there is no $q\in Q$ such that $q\leq p$.

We say $P$ is \textit{well-founded} if every nonempty subset of $P$ has a minimal element.
\end{definition}
\begin{lemma}\label{thm:upsetsantichains}
Let $(P,\,\leq)$ be a~poset. If $P$ is well-founded, then, for every upset $U \subseteq P$, there is an antichain $O\subseteq P$ such that $U = \upset_\leq O$.
\end{lemma}
\begin{proof}
Let $U$ be an upset in $P$ and let $O\subseteq U$ be the set of minimal elements of $U$. $O$ is nonempty since $P$ is well-founded and $O$ is an antichain since the~comparability of any two distinct elements would contradict the~minimality of one of them. We will prove $U =\upset_\leq O$.

Since $O\subseteq U$, we have $\upset_\leq O \subseteq\upset_\leq U = U$. Conversely, if $x\in U$, we consider the~set $U_x = \{u\in U\mid u\leq x\}$. This set is nonempty ($x\in U_x$) and thus has a~minimal element $m_x\in U_x$ since $P$ is well-founded. Moreover $m_x\in O$ -- if not, there is $m'\in U\setminus U_x$ such that $m'\lneq m_x$, but then $m'\lneq m_x\leq x$ and thus, by definition, $m'\in U_x$. This contradicts the minimality of~$m_x$. Hence, $x\in\upset_\leq O$ since $x\geq m_x\in O$, which concludes the proof of ``$\supseteq$''.
\end{proof}

\begin{THMdownsetsofNmTHM}
Let $m\in\N$ and let $\preceq$ be the product ordering on $\N_0^m$. The~poset $(\mathfrak{L}_{\downset_\preceq}(\N_0^{m}),\,\supseteq)$ satisfies ACC.
\end{THMdownsetsofNmTHM}
\begin{proof}
The~poset $(\mathfrak{L}_{\downset_\preceq}(\N_0^m),\,\supseteq)$ of downsets with dual ordering satisfies ACC if and only if the poset $(\mathfrak{L}_{\upset_\preceq}(\N_0^m),\,\subseteq)$ of upsets satisfies ACC. $\N_0^m$ is clearly well-founded, thus, by Lemma \ref{thm:upsetsantichains}, every upset is generated by some antichain in $\N^m$ with respect to $\upset_\preceq$. Moreover, $\N_0^m$ is countable, hence, we only need to show that every antichain of $(\N_0^m,\,\preceq)$ is finite due to Lemma \ref{thm:ACCFG}.

We will finish the proof by induction on $m$. Every antichain in $\N_0$ is finite since $\N_0$ is linearly ordered.

Assume that every antichain in $\N_0^{m-1}$ is finite (induction hypothesis). For the contradiction, let $O\subseteq \N_0^m$ be an~infinite antichain.

Let us pick arbitrary $(n_1,\dots,\,n_m)\in O$. Since $O$ is an antichain, for every element $(n_1',\dots,\,n_m')\in O\setminus\{(n_1,\dots,\,n_m)\}$ there is at least one $i\in E_m$ such that $n_i'\lneq n_i$. This observation lets us decompose $O$ as follows:
\begin{align*}
    O &= \{(n_1,\dots,\,n_m)\}\cup\bigcup_{i\,\in\,E_m} \{(n_1',\dots,\,n_m')\in O \mid n_i'\lneq n_i\} =\\
    &= \{(n_1,\dots,\,n_m)\}\cup\bigcup_{i\,\in\,E_m}\bigcup_{ j\,\in\, E_{n_i - 1}} \{(n_1',\dots,\,n_m')\in O \mid n_i' = j\}\,.
\end{align*}
This union is finite while $O$ is infinite; thus, there are $i\in E_m$ and $j\in E_{n_i - 1}$ such that the set $\{(n_1',\dots,\,n_m')\in O \mid n_i' = j\}$ is infinite. Then, the set
\begin{equation*}
    \{(n_1',\dots,\,n_{i-1}',\,n_{i+1}',\dots,\,n_m')\mid (n_1',\dots,\,n_{i-1}',\,j,\,n_{i+1}',\dots,\,n_m')\in O\}
\end{equation*}
is an infinite antichain in $\N_0^{m-1}$ which is the~desired contradiction.
\end{proof}



\end{document}